%% file: ConjugatorLength_17_FINAL_ArXiv.tex
\documentclass[11pt,oneside]{amsart}

\usepackage{amsmath,ifthen, amsfonts, amssymb,
srcltx, amsopn, color, enumerate, amsthm} 
\usepackage[linktocpage=true]{hyperref}
\usepackage[cmtip,arrow]{xy}
\usepackage{pb-diagram, pb-xy}
\usepackage{overpic}
\usepackage{xcolor}
\usepackage{epstopdf}

\usepackage[T1]{fontenc}

\newcommand{\showcomments}{no}

\newsavebox{\commentbox}

\long\def\Restate#1#2#3#4{
\medskip\par\noindent
{\bf #1 \ref{#2} #3} {\it #4}\par\medskip }

\newtheorem{thm}{Theorem}[section]
\newtheorem{lem}[thm]{Lemma}

\newtheorem{prop}[thm]{Proposition}

\newtheorem{thmi}{Theorem}

\newtheorem{questioni}[thmi]{Question}
\newtheorem{conji}[thmi]{Conjecture}

\theoremstyle{definition}
\newtheorem{defn}[thm]{Definition}
\newtheorem{rem}[thm]{Remark}

\newtheorem{exmp}[thm]{Example}

\newtheorem{notation}[thm]{Notation}

\newtheorem{claim}{Claim}

\newtheorem{claim*}{Claim}

\DeclareMathOperator{\diam}{\textup{\textsf{diam}}}

\newcommand{\neb}{\mathcal N}
\def\MCG{\mathcal{MCG}}

\makeatletter

\newcommand{\Rmnum}[1]{\mathbf{{\expandafter\@slowromancap\romannumeral #1@}}}

\makeatother

\newcommand{\mc}{\mathcal}

\def\nbhd{{\mathcal N}}

\let\oldmarginpar\marginpar
\renewcommand\marginpar[1]{\-\oldmarginpar[\raggedleft\footnotesize #1]{\raggedright\footnotesize #1}}

\newcommand{\tsh}[1]{\left\{\kern-.7ex\left\{#1\right\}\kern-.7ex\right\}}
\newcommand{\Tsh}[2]{\tsh{#2}_{#1}}
\newcommand{\ignore}[2]{\Tsh{#2}{#1}}

\newcommand{\co}{\colon}

\newcounter{enumitemp}

\newcommand{\dist}{\textup{\textsf{d}}}

\newcommand{\cuco}[1]{{\mathcal #1}}

\newcommand{\fontact}{{\mathcal C}}

\newcommand{\gate}{\mathfrak g}

 \usepackage{mathabx}
\newcommand{\propnest}{\sqsubsetneq}

\newcommand{\nest}{\sqsubseteq}
\newcommand{\orth}{\bot}
\newcommand{\transverse}{\pitchfork}

\newcommand{\relevant}{\mathbf{Rel}}
\newcommand{\X}{\mathcal X}
\newcommand{\s}{\mathfrak S}

\newcommand{\frakS}{\mathfrak S}

\newcommand{\FU}[1]{ \mathbf F_{#1}}
\newcommand{\EU}[1]{ \mathbf E_{#1}}

\definecolor{blue(pigment)}{rgb}{0.2, 0.2, 0.6}

\newcommand{\trans}{\pitchfork}

\newcommand{\B}{\operatorname{Big}}

\newcommand{\FixedF}{\mathbb{F}}

\title[Conjugator lengths in hierarchically hyperbolic groups]{Conjugator lengths in\\ hierarchically hyperbolic groups}

\author[C. Abbott]{Carolyn Abbott}
\thanks{Abbott was supported by NSF grants DMS-1803368 and DMS-2106906.}
\address{Brandeis University}
\email{carolynabbott@brandeis.edu}
\author[J. Behrstock]{Jason Behrstock}
\thanks{Behrstock was supported by NSF grant DMS-1710890.}
\address{Lehman College and The Graduate Center, CUNY}
\email{jason.behrstock@lehman.cuny.edu}

\begin{document}

\begin{abstract} In this paper we establish upper bounds on the 
	length of the shortest conjugator between pairs of infinite order elements in a 
	wide class of groups. We obtain a general result which applies to 
	all hierarchically 
	hyperbolic groups, a class which includes mapping class groups, right-angled 
	Artin groups, Burger--Mozes-type groups, most $3$--manifold groups, and many others. 
	In this setting we establish a linear bound on 
	the length of the shortest 
	conjugator for any pair of conjugate Morse elements. For a subclass 
	of these groups, including, in particular, all 
	virtually compact special groups, we prove a sharper 
	result by obtaining a linear bound 
	on the length of the shortest conjugator between a suitable power 
	of any pair of conjugate 
	infinite order elements. \end{abstract}

\maketitle

The conjugacy length function is the minimal function which bounds
the length of a shortest conjugator between any two conjugate 
elements of a given group, in terms of the sum of the word
lengths of the elements.  When a set of elements in a group has a
linear conjugacy length function, we say that set has the linear
conjugator property.  For any subset of a group satisfying the linear
conjugator property, and given two elements of that subset, there is an
exponential-time algorithm which determines whether or not the given
elements are conjugate.  One of Dehn's classic decision problems is
the Conjugacy Problem, which asks if there is an algorithm to
decide conjugacy given any pair of elements in a given group. 
Even in groups
where the Conjugacy Problem is unsolvable for arbitrary pairs of  
elements, establishing the linear conjugator property for a particular subset allows
one to solve the Conjugacy Problem for that subset.

An early established result about hyperbolic groups is they have the linear conjugacy property \cite{Lysenok}, 
thereby providing a quantitative certification of how complicated a  
conjugator needs to be. 
Exploiting the parallels between pseudo-Anosovs in the mapping class 
group and loxodromic elements in a hyperbolic group, Masur--Minsky 
proved the  analogous result  
that the set of pseudo-Anosov elements satisfies the linear conjugator property 
\cite{MasurMinsky:II}. 
These results beg the question of whether 
shortest conjugators of ``hyperbolic-like'' elements should be linear 
in the length of the elements being conjugated (see 
Conjecture~\ref{conj:morselcp} for a precise formulation).

In the presence of non-positive curvature, the linear conjugator
property is surprisingly common, as we show in this paper, extending
an already interesting class of known examples.  Previously 
established cases of the
linear conjugator property include: mapping class groups 
(established for pseudo-Anosovs in \cite{MasurMinsky:II}, generalized 
to all elements in 
\cite{Tao:LinConjMCG}; see also \cite{BehrstockDrutu:thick2} for a later, 
unified proof); 
hyperbolic elements in semi-simple
Lie groups \cite{Sale:LinConjLie}; arbitrary elements in lamplighter
groups \cite{Sale:GeomConjLamp}; non-peripheral elements in a 
relatively hyperbolic group \cite{Bumagin:relhyplinconj};  Morse elements in groups acting on
CAT(0) spaces, \cite{BehrstockDrutu:thick2}; and Morse elements in a
prime $3$--manifold \cite{BehrstockDrutu:thick2}. Additionally,
right-angled Artin groups enjoy the linear conjugator property; this
result is not explicitly stated in the literature, but it follows from
work in \cite{Servatius:AutRAAG} (and we give a new proof below).

In light of this, we will work in the general setting of
hierarchically hyperbolic groups, introduced by
Behrstock--Hagen--Sisto \cite{BehrstockHagenSisto:HHS1}. 
This class of groups is quite
large, encompassing many groups of interest, including: mapping class
groups \cite{BehrstockHagenSisto:HHS2}; right-angled Artin groups, and
more generally fundamental groups of compact special cube complexes
\cite{BehrstockHagenSisto:HHS1} and other CAT(0) cube complexes
\cite{HagenSusse:HHScubical}; $3$--manifold groups with no Nil or Solv
components \cite{BehrstockHagenSisto:HHS2}; and lattices in products
of trees, i.e., as constructed by Burger--Mozes, Wise, and others, see
\cite{BehrstockHagenSisto:HHS1,
BurgerMozes:TreeProducts,BurgerMozes:TreeProductsIHES,
Caprace:indiscrete,JanzenWise, Rattaggi, Wise:CSC}.  There are a
number of other examples, as well, for instance groups obtained from
combination theorems, including taking graphs of hierarchically
hyperbolic groups and graph products of hierarchically hyperbolic
groups \cite{BehrstockHagenSisto:HHS2, Spriano:GraphsHHG,
BerlaiRobbio:combinationHHG}, or by taking certain quotients of a
hierarchically hyperbolic group \cite{BehrstockHagenSisto:asdim}.

The first theorem is new for most hierarchically
hyperbolic groups; it also provides a unified proof for 
the previously known cases. An element in a finitely generated group is called
\emph{Morse} if its orbit in the group is a quasi-geodesic with the
property that any $(\lambda,c)$--quasi-geodesic beginning and ending on this
orbit is completely contained within a uniformly bounded neighborhood
of this orbit.  We note that Morse elements in a group are ones whose
geometry in the Cayley graph is similar to that of the axis of a
loxodromic isometry of a hyperbolic space (via the Morse Lemma); in a
hierarchically hyperbolic group the Morse elements can be
characterized in several equivalent ways, see
\cite[Theorem~B]{AbbottBehrstock:HHSlargest}.

\begin{thmi} \label{mainthm}
Let $(G,\s)$ be a hierarchically hyperbolic group.  There exist
constants $K,C$ such that if $a,b\in G$ are Morse elements 
which are conjugate in $G$, then there exists $g\in G$ with
$ga=bg$ and
$$
|g|\leq K(|a|+|b|)+C.
$$  
\end{thmi}

One special case of the above theorem is a new proof that conjugate 
pseudo-Anosov elements in the mapping class group have a linear bound 
on the length of their shortest conjugator; this case was the main theorem of 
\cite{MasurMinsky:II}.

A natural conjecture arising from Theorem~\ref{mainthm} is the 
following generalization:
\begin{conji}\label{conj:morselcp} In a finitely generated group, the set of Morse elements 
	satisfy the linear conjugator property.
\end{conji}

Understanding exactly how the linear conjugator property is related to to 
hyperbolic properties in a group remains a rich question, and with 
Theorem~\ref{mainthm}, hierarchically hyperbolic groups provide a 
good place to study this. For instance, we conjecture that there 
exist 
hierarchically hyperbolic groups where the conjugacy length function 
is exponential. Accordingly, we don't believe the 
linear conjugator property holds for all elements in all 
hierarchically hyperbolic 
groups, but it does in a number of important examples, which leads us 
to ask:

\begin{questioni}\label{quest:whichHHGlcp} Under what conditions does a hierarchically hyperbolic group
	satisfy the linear conjugator property for all elements?
\end{questioni}

\medskip 

In Section~\ref{sec:shhg} we introduce a family of
hierarchically hyperbolic groups in which the notion of
orthogonality carries with it not just geometric implications, but also a 
useful algebraic structure.  
The way in which the 
algebraic structure is related to
orthogonality in these groups generalizes the relationship  
between commutativity and 
orthogonality in mapping class groups and compact special 
groups. 
This family is defined through a
series of conditions called the $\mathbf F_U$ stabilizers, orthogonal
decomposition, and commutative properties (see Section~\ref{sec:shhg} for
the precise definitions).

After showing in Proposition~\ref{prop:specialHHGs} 
that many groups satisfy the properties we introduce 
and that being in this family is preserved by various 
combination theorems, we then study conjugators in these groups. 
The following generalizes Theorem~\ref{mainthm} by 
removing the hypothesis that the elements are Morse:

\begin{thmi}\label{thmi:SHHGconjugators}
Let $(G,\s)$ be a
hierarchically hyperbolic group satisfying the $\mathbf F_U$
stabilizers, orthogonal decomposition, and commutative properties.
There exist constants $K,C$ and $N$ such that if $a,b\in G$ are
infinite order elements which are conjugate in $G$, then there exists $g\in G$ with $ga^N=b^Ng$ and
$$
|g|\leq K(|a|+|b|)+C.
$$ 
\end{thmi}

In particular, compact special groups (i.e., fundamental groups of
compact cube complexes which are special in the sense of Haglund--Wise
\cite{HaglundWiseSpecial}) satisfy the $\mathbf F_U$ stabilizers,
orthogonal decomposition, and commutative properties.  Therefore
Theorem \ref{thmi:SHHGconjugators} holds for all virtually compact
special groups.  We note that \cite{CrispGodelleWiest:ConjRAAGs} establish a linear 
time solution to the conjugacy problem for fundamental groups of 
compact special cube complexes. Their result doesn't a priori establish the 
linear conjugator property of Theorem~\ref{thmi:SHHGconjugators}, 
although we believe that their approach could be used to do so. 

We believe that the linear conjugator property will
in general fail for cubulated groups without the hypothesis that the
cube complex is special.  Our proof relies heavily on the close
relationship between orthogonality and commutation, something which
can fail for CAT$(0)$ cubical groups which are not special, even
though they may be hierarchically hyperbolic groups.  The Burger-Mozes
groups \cite{BurgerMozes:TreeProducts,BurgerMozes:TreeProductsIHES},
for instance, are plausibly a counterexample; see
\cite[Section~8.2.2]{BehrstockHagenSisto:HHS1}, or Wise's construction
\cite{Wise:CSC}.

\subsection*{Acknowledgments}

We thank Mark Hagen for many interesting discussions about 
hierarchically hyperbolic spaces.  We thank
Jacob Russell and Abdul Zalloum for feedback on an early draft, and the anonymous 
referees for useful comments.

\section{Background} \label{sec:background}

\subsection{Hyperbolic geometry} We begin by gathering several facts about $\delta$--hyperbolic metric spaces and refer the reader to \cite{BridsonHaefliger} for further details.

A map of metric spaces $f\colon (X,\dist_X)\to (Y,\dist_Y)$ is a \emph{$(\lambda,c)$--quasi-isometric embedding} if for all $x,y\in X$
\[
\frac1\lambda\dist_X(x,y)-c\leq \dist_Y(f(x),f(y))\leq \lambda\dist_X(x,y)+c.
\]
A \emph{$(\lambda,c)$--quasi-geodesic}  is a \emph{$(\lambda,c)$--quasi-isometric embedding} of an interval $I\subseteq \mathbb R$ into $X$, and a \emph{geodesic} is an isometric embedding of $I$ into $X$. In both cases, we allow $f$ to be a \emph{coarse map}, that is, a map which sends points in $I$ to uniformly bounded diameter sets in $X$.  A (coarse) map $f\colon [0,T]\to X$ is an \emph{unparametrized $(\lambda,c)$--quasi-geodesic} if there exists a strictly increasing function $g\colon [0,T']\to[0,T]$ such that the following hold:
	\begin{itemize} 
	\item $g(0)=f(0)$, 
	\item $g(T')=f(T)$,  
	\item $f\circ g\colon [0,T']\to X$ is a $(\lambda,c)$--quasi-geodesic, and 
	\item for each $j\in[0,T']\cap \mathbb N$, we have the diameter of $f(g(j))\cup f(g(j+1))$ is at most $c$.
	\end{itemize}
	
	If $Y\subseteq X$ is a subspace, then for any constant $K\geq 0$,  we denote the closed $K$--neighborhood of $Y$ in $X$ by
\[\mathcal N_K(Y)=\{x\in X\mid \dist_X(x,Y)\leq K\}.\]   We may write $\nbhd^X_K(Y)$ to emphasize that the neighborhood is being taken in $X$.
	
A subspace $Y\subseteq X$ is \emph{$\sigma$--quasi-convex} if any geodesic in $X$ with endpoints in $Y$ is contained in $\nbhd_\sigma(Y)$.  The subspace $Y$ is called \emph{quasi-convex} if it is $\sigma$--quasi-convex for some $\sigma$.

If $X$ is a geodesic metric space and $x,y\in X$, we let $[x,y]$ denote a geodesic from $x$ to $y$.  If we want to emphasize the metric space $X$, we write $[x,y]_X$.

\begin{defn}[$\delta$--hyperbolic space]
Fix $\delta\geq 0$.  A metric space $X$ is \emph{$\delta$--hyperbolic} if given any $x,y,z\in X$ and any geodesics $\alpha,\beta,\gamma$ between them, we have $\alpha\cup\beta\subseteq\nbhd_\delta(\gamma)$. If the particular choice of $\delta$ is not important, we simply say that $X$ is \emph{hyperbolic}.
\end{defn}

Quasi-geodesics in a hyperbolic spaces satisfy two important properties: a local-to-global property and the Morse Lemma.  A path $p$ is an \emph{$L$--local $(\lambda,c)$--quasi-geodesic} if every subpath $p$ of length at most $L$ is a $(\lambda,c)$--quasi-geodesic.
\begin{lem}[Local-to-global Property]\label{lem:localtoglobal}
Let $X$ be a $\delta$--hyperbolic metric space and fix $\ell_0\geq 0$.
There exists $L=L(\ell_0,\delta)$ depending only on $\delta$ and
$\ell_0$ such that:  if $\ell\in[0,\ell_0]$ and
$\gamma\colon I\to X$ is an $L$--local $(1,\ell)$-quasi-geodesic, then
$\gamma$ is a global $(2,\ell)$--quasi-geodesic.
\end{lem}

\begin{lem}[Morse Lemma]\label{lem:Morse}
Let $X$ be a $\delta$--hyperbolic metric space, and fix $\lambda\geq 1$ and $c\geq0$.  There exists a constant $\sigma$ depending only on $\delta,\lambda$, and $c$ such that if $\gamma_1$ and $\gamma_2$ are $(\lambda,c)$--quasi-geodesics in $X$ with the same endpoints, then $\gamma_1\subseteq \nbhd_\sigma(\gamma_2)$.
\end{lem}
We say $\sigma$ is the \emph{Morse constant} associated to $( 
\lambda,c)$--quasi-geodesics in a $\delta$--hyperbolic~space.

\medskip

Let $G$ act by isometries on a $\delta$--hyperbolic metric space $X$.  Then  $h\in G$ is: 
	\begin{itemize}
	\item \emph{elliptic} if it has bounded orbits; 
	\item \emph{loxodromic} if the map $\mathbb Z\to X$ defined by $n\mapsto h^nx$ is a quasi-isometric embedding for some (equivalently, any) $x\in X$;
	\item \emph{parabolic} otherwise.
	\end{itemize}
Isometries of a hyperbolic space can also be characterized by their limit sets in the Gromov boundary $\partial X$ of $X$.  An element $h\in G$ is elliptic, parabolic, or loxodromic if the limit set of $h$ has cardinality $0,1$, or $2$, respectively.  If the limit set of $h$ has cardinality 2, we call these limit points $h^{\pm\infty}$.

Loxodromic isometries will play a particularly important role in this
paper, and we discuss them in more depth.  For the rest of the
subsection, fix a group $G$ acting by isometries on a
$\delta$--hyperbolic space $X$, and fix an element $h\in G$ that is
loxodromic with respect to this action.  The \emph{translation length}
of $h$ is $[h]_X:=\inf_{x\in X}\dist_X(x,hx)$, or simply $[h]$ if the
space $X$ is clear.  The \emph{stable translation length} of $h$ in
$X$ is
\[
\tau_X(h):=\lim_{n\to\infty}\frac{\dist_X(x_0,h^nx_0)}{n}
\]
for some (equivalently, any) $x_0\in X$.  These two quantities are related by $\tau_X(h)\leq [h]_X\leq \tau_X(h)+16\delta$.

The element $h$ acts on $X$ as translation
along a quasi-geodesic axis which connects the  two limit points
$h^{\pm\infty}$ of $h$ in $\partial X$.  Up to passing to powers, such
an axis can be chosen to be a \emph{uniform quality} quasi-geodesic,
that is, with quasi-geodesic constants which depending only on $\delta$ and not
on the choice of loxodromic isometry.  As this will be important in
this paper, we now describe the construction of such an axis.  The
following lemma summarizes results in \cite[Section~3]{Coulon}.

\begin{lem}[Construction of an $\ell$--nerve]\label{lem:ellnerve}
Let $G$ act on a $\delta$--hyperbolic space $X$.  Suppose $h\in G$ is
loxodromic and $\tau_X(h)\geq L_S\delta-16\delta$, where $L_S$ 
depends only on $\delta$ (and is more explicitly described in \cite[Definition~2.8]{Coulon}).  Then for any $\ell\in [0,\delta]$, there exists a
$(2,\ell)$--quasi-geodesic $\gamma_h^X$ in $X$ which connects the
limit points $h^{\pm\infty}$ of $h$ in $\partial X$, called the
\emph{$\ell$--nerve} of $h$.  The $\ell$--nerve of $h$ is
$(\ell+8\delta)$--quasiconvex and preserved by $h$.
\end{lem}

We briefly recall the construction of the $\ell$--nerve and refer the
reader to \cite[Definition~3.3 and subsequent remark]{Coulon} for further details. 
Fix $\ell\in[0,\delta]$.  By
the definition of the translation length $[h]$, there exists $x\in X$
such that $\dist_X(x,hx)\leq [h]+\ell/2$.  Thus we can find a
$(1,\ell/2)$--quasi-geodesic $\gamma$ from $x$ to $hx$.  If $\gamma$
has length $T$, then $[h]\leq T<[h]+\ell/2$.  Extend $\gamma$ to a
bi-infinite path $\gamma_h^X$ using the action of $\langle h\rangle$;
that is $\gamma_h^X$ is the concatenation of the segments $h^i\gamma$
for $i\in\mathbb Z$.
In particular, for any $h$ for which $\tau_X(h)\geq
L_S\delta-16\delta$, we have that $\gamma$ is an $L_S\delta$--local
$(1,\ell)$--quasi-geodesic.  By Lemma \ref{lem:localtoglobal},
$\gamma_h^X$ is therefore a global $(2,\ell)$--quasi-geodesic.

\begin{defn}[Quasi-geodesic axis]\label{def:qgeoaxis}
Let $G$ act on a $\delta$--hyperbolic space $X$, and fix the constant
$L_S$ from Lemma~\ref{lem:ellnerve} and a constant $\ell\in
[0,\delta]$.  Suppose $h\in G$ is a loxodromic isometry of $X$, and
let $k\in \mathbb N$ be such that $\tau_X(h^k)\geq
L_S\delta-16\delta$.  A \emph{$(2,\ell)$--quasi-geodesic axis} of
$h^k$ in $X$ is an $\ell$--nerve $\gamma_{h^k}^X$.  If the
quasi-geodesic constants are not important, we simply call
$\gamma_{h^k}^X$ an \emph{axis} of $h^k$.
\end{defn}

Suppose $h,g\in G$ and $h$ is loxodromic with respect to the action of $G$ on a hyperbolic metric space $X$ with $\tau_X(h)\geq L_S\delta-16\delta$ and a $(2,\ell)$--quasi-geodesic axis $\gamma_h^X$.   Then $ghg^{-1}$ is also loxodromic with respect to the action of $G$ on $X$ with translation length $\tau_X(ghg^{-1})=\tau_X(h)$, and it follows from the construction of the $\ell$--nerve that $g\gamma_h^X$ is a $(2,\ell)$--quasi-geodesic axis of $ghg^{-1}$.

\subsection{Hierarchically hyperbolic spaces}

We recall the definition of a hierarchically hyperbolic space as 
given in \cite{BehrstockHagenSisto:HHS2}.    The definition is in the setting of a \emph{quasi-geodesic metric space}, that is, a metric space in which any two points can be connected by a uniform quality quasi-geodesic.

\begin{defn}[Hierarchically hyperbolic space]\label{defn:HHS}
The quasi-geodesic space $(\cuco X,\dist_{\cuco X})$ is a
\emph{hierarchically hyperbolic space} (HHS) if there exists
$\delta\geq0$, an index set $\mathfrak S$, and a set $\{\fontact
W:W\in\mathfrak S\}$ of $\delta$--hyperbolic spaces $(\fontact
W,\dist_W)$, satisfying the following conditions:

\begin{enumerate}
\item\textbf{(Projections.)}\label{item:dfs_curve_complexes} There is
a set $\{\pi_W\co \cuco X\rightarrow2^{\fontact W}\mid W\in\mathfrak
S\}$ of \emph{projections} sending points in $\cuco X$ to sets of
diameter bounded by some $\xi\geq0$ in the various $\fontact
W\in\mathfrak S$.  Moreover, there exists $K$ so that each $\pi_W$ is
$(K,K)$--coarsely Lipschitz and $\pi_W(\cuco X)$ is $K$--quasiconvex
in $\fontact W$.

 \item \textbf{(Nesting.)} \label{item:dfs_nesting} $\mathfrak S$ is
 equipped with a partial order $\nest$, and either $\mathfrak
 S=\emptyset$ or $\mathfrak S$ contains a unique $\nest$--maximal
 element; when $V\nest W$, we say $V$ is \emph{nested} in $W$.  (We
 emphasize that $W\nest W$ for all $W\in\mathfrak S$.)  For each
 $W\in\mathfrak S$, we denote by $\mathfrak S_W$ the set of
 $V\in\mathfrak S$ such that $V\nest W$.  Moreover, for all $V,W\in\mathfrak S$
 with $V\propnest W$ there is a specified subset
 $\rho^V_W\subset\fontact W$ with $\diam_{\fontact W}(\rho^V_W)\leq\xi$.
 There is also a \emph{projection} $\rho^W_V\colon \fontact
 W\rightarrow 2^{\fontact V}$.  We call the elements of the index set $\s$ \emph{domains}.

 \item \textbf{(Orthogonality.)} 
 \label{item:dfs_orthogonal} $\mathfrak S$ has a symmetric and
 anti-reflexive relation called \emph{orthogonality}: we write $V\orth
 W$ when $V,W$ are orthogonal.  Also, whenever $V\nest W$ and $W\orth
 U$, we require that $V\orth U$.  We require that for each
 $T\in\mathfrak S$ and each $U\in\mathfrak S_T$ for which
 $\{V\in\mathfrak S_T\mid V\orth U\}\neq\emptyset$, there exists $W\in
 \mathfrak S_T-\{T\}$, so that whenever $V\orth U$ and $V\nest T$, we
 have $V\nest W$.  The domain $W$ is called the \emph{container}
 associated to $U$ in $T$. 
 Finally, if $V\orth W$, then $V,W$ are not
 $\nest$--comparable.
 
 \item \textbf{(Transversality and consistency.)}
 \label{item:dfs_transversal} If $V,W\in\mathfrak S$ are not
 orthogonal and neither is nested in the other, then we say $V,W$ are
 \emph{transverse}, denoted $V\transverse W$.  There exists
 $\kappa_0\geq 0$ such that if $V\transverse W$, then there are
  sets $\rho^V_W\subseteq\fontact W$ and
 $\rho^W_V\subseteq\fontact V$ each of diameter at most $\xi$ and 
 satisfying: $$\min\left\{\dist_{
 W}(\pi_W(x),\rho^V_W),\dist_{
 V}(\pi_V(x),\rho^W_V)\right\}\leq\kappa_0$$ for all $x\in\cuco X$.
 
 For $V,W\in\mathfrak S$ satisfying $V\nest W$ and for all
 $x\in\cuco X$, we have: $$\min\left\{\dist_{
 W}(\pi_W(x),\rho^V_W),\diam_{\fontact
 V}(\pi_V(x)\cup\rho^W_V(\pi_W(x)))\right\}\leq\kappa_0.$$ 
 
 The preceding two inequalities are the \emph{consistency inequalities} for points in $\cuco X$.
 
 Finally, if $U\nest V$, then $\dist_W(\rho^U_W,\rho^V_W)\leq\kappa_0$ whenever $W\in\mathfrak S$ satisfies either $V\propnest W$ or $V\transverse W$ and $W\not\orth U$.
 
 \item \textbf{(Finite complexity.)} \label{item:dfs_complexity} There exists $n\geq0$, the \emph{complexity} of $\cuco X$ (with respect to $\mathfrak S$), so that any set of pairwise--$\nest$--comparable elements has cardinality at most $n$.
  
 \item \textbf{(Large links.)} \label{item:dfs_large_link_lemma} There
exist $\lambda\geq1$ and $E\geq\max\{\xi,\kappa_0\}$ such that the following holds.
Let $W\in\mathfrak S$ and let $x,x'\in\cuco X$.  Let
$N=\lambda\dist_{_W}(\pi_W(x),\pi_W(x'))+\lambda$.  Then there exists $\{T_i\}_{i=1,\dots,\lfloor
N\rfloor}\subseteq\mathfrak S_W-\{W\}$ such that for all $T\in\mathfrak
S_W-\{W\}$, either $T\in\mathfrak S_{T_i}$ for some $i$, or $\dist_{
T}(\pi_T(x),\pi_T(x'))<E$.  Also, $\dist_{
W}(\pi_W(x),\rho^{T_i}_W)\leq N$ for each $i$.

 \item \textbf{(Bounded geodesic image.)}
 \label{item:dfs:bounded_geodesic_image} There exists $E>0$ such that 
 for all $W\in\mathfrak S$,
 all $V\in\mathfrak S_W-\{W\}$, and all geodesics $\gamma$ of
 $\fontact W$, either $\diam_{\fontact V}(\rho^W_V(\gamma))\leq E$ or
 $\gamma\cap\neb_E(\rho^V_W)\neq\emptyset$.
 
 \item \textbf{(Partial Realization.)} \label{item:dfs_partial_realization} There exists a constant $\alpha$ with the following property. Let $\{V_j\}$ be a family of pairwise orthogonal elements of $\mathfrak S$, and let $p_j\in \pi_{V_j}(\cuco X)\subseteq \fontact V_j$. Then there exists $x\in \cuco X$ so that:
 \begin{itemize}
 \item $\dist_{V_j}(\pi_{V_j}(x),p_j)\leq \alpha$ for all $j$,
 \item for each $j$ and 
 each $V\in\mathfrak S$ with $V_j\nest V$, we have 
 $\dist_{V}(\pi_V(x),\rho^{V_j}_V)\leq\alpha$, and
 \item if $W\transverse V_j$ for some $j$, then $\dist_W(\pi_W(x),\rho^{V_j}_W)\leq\alpha$.
 \end{itemize}

\item\textbf{(Uniqueness.)} For each $\kappa\geq 0$, there exists
$\theta_u=\theta_u(\kappa)$ such that if $x,y\in\cuco X$ and
$\dist_{\cuco X}(x,y)\geq\theta_u$, then there exists $V\in\mathfrak S$ such
that $\dist_V(\pi_V(x),\pi_V(y))\geq \kappa$.\label{item:dfs_uniqueness}
\end{enumerate}

\end{defn}

For ease of readability, given $U\in\mathfrak S$, we typically 
suppress the projection map $\pi_U$ when writing distances in
$\fontact U$, i.e., given $x,y\in\cuco X$ and $p\in\fontact U$ we
write $\dist_U(x,y)$ for $\dist_U(\pi_U(x),\pi_U(y))$ and
$\dist_U(x,p)$ for $\dist_U(\pi_U(x),p)$.  When necessary for clarity, we may also write $\fontact(U)$ instead of $\fontact U$.

An important consequence of being a hierarchically hyperbolic space is the following distance formula, which relates distances in $\X$ to distances in the hyperbolic spaces $\fontact U$ for $U\in\s$.  Give $a,b\in \mathbb R$, the notation $\ignore{a}{b}$ denotes the quantity which is $a$ if $a\geq b$ and is $0$ otherwise.  Given $C,D$, we say $a\asymp_{C,D} b$ if $C^{-1} a-D\leq b\leq Ca+D$.  We use $a\asymp_D b$ if $|a-b|\leq D$, and we use $a\preceq_{C,D} b$ if $a\leq Cb+D$.

\begin{thm}[Distance formula for HHS;  
    \cite{BehrstockHagenSisto:HHS2}]\label{thm:distance_formula}
Let $(\cuco X, \mathfrak S)$ be a hierarchically hyperbolic space.  Then
there exists $s_0$ such that for all $s\geq s_0$, there exist $C,D$ so
that for all $x,y\in\cuco X$,
$$\dist_{\X}(x,y)\asymp_{C,D}\sum_{U\in\mathfrak
S}\ignore{\dist_U(x,y)}{s}.$$
\end{thm}

The distance formula  says that the distance between two points in $\X$ can be approximated by measuring the distances between their projections to the hyperbolic spaces, and, moreover, that we only need to consider hyperbolic spaces for which that projection is sufficiently large.    

\begin{defn}[Relevant domains] For any constant $R\geq s_0$ and any two points $x,y\in \mathcal X$, we say $U\in\s$ is \emph{relevant} (with respect to $x,y,R$) if $\dist_U(x,y)\geq R$; if we want to emphasize the constant $R$, we say that $U$ is $R$--\emph{relevant} (with respect to $x,y$).  We denote  the set of $R$--relevant domains by $\relevant(x,y;R)$.
\end{defn}

In other words, the set of $R$--relevant domains for a pair of points $x,y\in\X$ are the domains which appear in the distance formula for $x$ and $y$ with the threshold  $s=R$. 

\begin{notation} \label{not:E}
Given a hierarchically hyperbolic space $(\X,\s)$ we let $E$ denote a 
    constant larger than 
any of the constants occurring in Definition \ref{defn:HHS} and larger than the constant $s_0$ from Theorem \ref{thm:distance_formula}.
\end{notation}

\begin{defn}[Hierarchy path] 
    Given a hierarchically hyperbolic space $(\X,\s)$ 
    and a constant $\lambda\geq 1$, a 
    \emph{$(\lambda,\lambda)$--hierarchy
    path} $\gamma\subset \X$ is a $(\lambda,\lambda)$--quasi-geodesic in $\X$ with the 
    property that for each $U\in\s$ the path $\pi_{U}(\gamma)$ is an 
    unparametrized $(\lambda,\lambda)$--quasi-geodesic in $\fontact U$. 
    \end{defn}
    By 
	\cite[Theorem~4.4]{BehrstockHagenSisto:HHS2}, for any 
	sufficiently large $\lambda$, any two points $x,y\in\X$ are connected 
	by a 
	$(\lambda,\lambda)$--hierarchy path.  We fix such a constant $\lambda>E$, and let $\mu(x,y)\subseteq \X$ denote a $(\lambda,\lambda)$--hierarchy path from $x$ to $y$.

\begin{defn} [Hierarchically hyperbolic group]\label{defn:HHG}
    A finitely generated group $G$ is  a  \emph{hierarchically 
    hyperbolic group} if some (hence any) Cayley graph of $G$ is a hierarchically hyperbolic space, and the hierarchically hyperbolic structure is $G$--invariant.  In particular, a hierarchically hyperbolic group is a finitely generated group $G$, equipped with a specific choice of finite generating set, such that there is a hierarchically hyperbolic space $(G,\s)$ satisfying the following properties:
    \begin{itemize}
    \item $G$ acts cofinitely on $\s$, preserving the relations $\nest,\trans$ and $\orth$;
    \item For each $U\in\s$ and $g\in G$, there is an isometry $g\colon \fontact U\to\fontact(gU)$, and if $h\in G$, then the isometry $gh\colon \fontact U\to\fontact(ghU)$ is equal to the composition $\fontact U\xrightarrow{h}\fontact (hU)\xrightarrow{g}\fontact(ghU)$;
    \item For each $U\in \s$ and $g,x\in G$, we have $g\pi_U(x)=\pi_{gU}(gx)$; and
    \item For each $U,V\in\s$ such that $U\trans V$ or $U\propnest V$ and each $g\in G$, we have $\rho^{gU}_{gV}=g\rho^U_V$.
    \end{itemize}
 \end{defn}   
 
 Given a hierarchically hyperbolic group $(G,\s)$, we use $\dist_G$ to denote the distance in the group $G$ with respect to some (fixed) finite generating set.

\subsection{Gate Maps and Standard Product Regions}

In analogy with quasiconvex subspaces of hyperbolic spaces, there is a notion of a hierarchically quasiconvex subspace of a hierarchically hyperbolic space $\X$.

\begin{defn}[Hierarchically quasiconvex]
Let $(\X,\s)$ be a hierarchically hyperbolic space.  A subspace $\mathcal Y$ of $\X$ is \emph{$k$--hierarchically quasiconvex} for some $k\colon[0,\infty)\to[0,\infty)$ if the following hold:
\begin{enumerate}
\item For all $U\in \s$, the projection $\pi_U(\mathcal Y)$ is a $k(0)$--quasiconvex subspace of $\fontact U$;

\item For every $\kappa>0$ and every point $x\in \X$ satisfying $\dist_U(\pi_U(x),\pi_U(\mathcal Y))\leq \kappa$ for all $U\in\s$, we have $\dist_{\X}(x,\mathcal Y)\leq k(\kappa)$.

\end{enumerate}
\end{defn}

The first condition says that the subspace $\mathcal Y$ projects to a (uniformly) quasiconvex subspace in every hyperbolic space, while the second condition ensures that all points in $\X$ which project near $\mathcal Y$ in every hyperbolic spaces are near $\mathcal Y$ in $\X$.

As is the case for quasiconvex subspaces of hyperbolic spaces, if $\mathcal Y$ is a hierarchically quasiconvex subspace of a hierarchically hyperbolic space $\X$, then there is a well-defined ``nearest point projection'' from $\X$ to $\mathcal Y$, called a gate map.

\begin{defn}[Gate maps] \label{def:gates}
If $(\X,\s)$ is a hierarchically hyperbolic group and $\mathcal Y$ is
a hierarchically quasiconvex subspace of $\X$, then the \emph{gate
map} is a coarsely-Lipschitz map $\mathfrak g_\mathcal Y\colon \X\to
2^\mathcal Y$, so that for each $x\in \X$, the image $\mathfrak
g_\mathcal Y(x)$ is a subset of the points in $\mathcal Y$ with the
property that for each $U\in\s$ the set $\pi_{U}(\mathfrak g_\mathcal
Y(x))$ uniformly coarsely coincides  
with the
closest point projection in $\fontact U$ of $\pi_{U}(x)$ to
$\pi_{U}(\mathcal Y)$.
\end{defn}
The following lemma shows that gate maps are uniformly coarsely equivariant.

\begin{lem}[{\cite[Lemma~4.16]{RussellSprianoTran:convexity}}] \label{lem:gatesequivar}
Let $(G,\s)$ be a hierarchically hyperbolic group, and let $\mathcal Y$ be a $k$--hierarchically quasiconvex subspace of $G$.  Then there exists a constant $A$ depending on $(G,\s)$ and $k$ such that for every $g,x\in G$, we have
\[
\dist_G(g\frak g_{\mathcal Y}(x),\frak g_{g\mathcal Y}(gx))\leq A.
\]
\end{lem}

We now recall an important family of hierarchically quasiconvex subspaces in a
hierarchically hyperbolic space called \emph{standard product regions}
introduced in \cite[Section~13]{BehrstockHagenSisto:HHS1} and studied
further in \cite{BehrstockHagenSisto:HHS2}.  The definition we give can be found in \cite[Definition~2.20]{Russell:relative} and is also discussed in \cite[Section~1.2.1]{BehrstockHagenSisto:asdim} .

\begin{defn}[Standard product region] \label{def:productregion}
Let $(\X,\s)$ be a hierarchically hyperbolic space, and let $U\in \s$. The \emph{standard product region for $U$} is the set
\[\mathbf P_U=\{x\in\X\mid \dist_V(x,\rho^U_V)\leq E \textrm{ for all $V\in \s$ with $V\trans U$ or $V\sqsupsetneq U$}\}.\]
Note that if $S\in\s$ is $\nest$--maximal, then $\mathbf P_S=\X$.
\end{defn}

In other words, given $U\in\s$ and $V\in \s$ satisfying $V\trans U$ or $V\sqsupsetneq U$, the product region $\mathbf P_U$ is precisely the set of points which project near $\rho^U_V$ in $\fontact V$.
It thus follows from this definition that for such $U,V$, we have  $\rho^U_V\asymp_E \pi_V(\mathbf P_U)$; that is, the projection $\rho^U_V$ is coarsely equal to the projection of the product region $\mathbf P_U\subseteq \X$ into $\fontact V$.

Though it is not obvious from this definition, the product region $\mathbf P_U$ is quasi-isometric to a space with decomposes as a direct product of two factors, $\mathbf F_U$ and $\mathbf E_U$.  As these factors will be important in this paper, we describe them in detail.  See \cite[Section~5.2]{BehrstockHagenSisto:HHS2} for additional details.  We first define $\FU U$ and $\EU U$ as abstract spaces.  In the paragraphs following the definitions, we explain that these spaces admit embeddings into $\X$.  Unless otherwise noted, we will always think of these embeddings, rather than the abstract spaces themselves.

\begin{defn}[Nested partial tuple ($\FU U$)]\label{defn:nested_partial_tuple}
Let $\mathfrak S_U=\{V\in\mathfrak S \mid V\nest U\}$.  Fix
$\kappa\geq E$ and let $\FU U$ be the set of
$\kappa$--consistent tuples in $\prod_{V\in\mathfrak S_U}2^{\fontact
V}$ (i.e., tuples satisfying the consistency inequalities  
of Definition~\ref{defn:HHS}.(\ref{item:dfs_transversal})).
\end{defn}

\begin{defn}[Orthogonal partial tuple ($\mathbf E_U$)]\label{defn:orthogonal_partial_tuple}
Let $\mathfrak S_U^\orth=\{V\in\mathfrak S\mid V\orth U\}\cup\{W\}$, where
$W$ is a $\nest$--minimal element such that $V\nest W$ for all
$V\orth U$.  Fix $\kappa\geq E$, and let $\mathbf E_U$ be the set of
$\kappa$--consistent tuples in $\prod_{V\in\mathfrak
S_U^\orth-\{A\}}2^{\fontact V}$.
\end{defn}

\begin{rem}
 The particular choice of constant $\kappa$ will not be important in this paper.  For simplicity, given a hierarchically hyperbolic group, we fix $\kappa=E$, and for each domain $U$ we consider only spaces $\mathbf F_U$ and $\mathbf E_U$ defined using $E$--consistent tuples. 
\end{rem}

Given $\cuco X$ and $U\in\mathfrak S$, there is a well-defined map $\phi_U\co\mathbf F_{U}\times \mathbf E_U\to\cuco X$.   The precise definition of this map is not necessary for this paper; we refer the interested reader to \cite[Construction~5.10]{BehrstockHagenSisto:HHS2}.  The product region $\mathbf P_U$ defined in Definition \ref{def:productregion} is coarsely equal to the image $\phi_U\left(\mathbf F_{U}\times \mathbf E_U\right)$ in $\cuco X$.  In this paper, we will  only work with  $\mathbf P_U$ and $\mathbf F_U$.  For all results that we state for $\mathbf F_U$,  analogous statements also hold for $\mathbf E_U$.

Fixing any $e\in \mathbf E_U$ restricts $\phi_U$ to a map $\phi^\nest\co \mathbf F_U\times \{e\}\to \X$.  In general this map $\phi^\nest$ depends on the choice of $e\in \mathbf E_U$.  When the basepoint is immaterial (or understood), we abuse notation and consider $\mathbf F_U$ to be a subspace of $\cuco X$, that is, $\mathbf F_U=\operatorname{im}\phi^\nest$.

It is proven in \cite[Lemma~5.5]{BehrstockHagenSisto:HHS2} that standard product regions $\mathbf P_U$ and their factors $\mathbf F_U\times \{e\}$ for each $e\in\mathbf E_U$  (considered as subspaces of $\cuco X$) are uniformly hierarchically quasiconvex.  Therefore there are well-defined gate maps $\frak g_{\mathbf P_U}\colon\X\to\mathbf P_U$ and $\frak g_{\mathbf F_U\times\{e\}}\colon\X\to\mathbf F_U\times \{e\}$  for each $U\in \s$ and  each $e\in \mathbf E_U$. 

\begin{rem}\label{rem:gates}
We note that the gate map $\frak g_{\mathbf F_U\times\{e\}}$ depends on the choice of $e\in \mathbf E_U$.  However, the image of the gate map in $\mc CV$ for any $V \nest U$ is independent of this choice (see \cite[Remark~1.16]{BehrstockHagenSisto:asdim}).  That is, if $e,e'\in\mathbf E_U$, then for any $x\in \X$, we have $\pi_V(\frak g_{\mathbf F_U\times\{e\}}(x))=\pi_V(\frak g_{\mathbf F_U\times\{e'\}}(x))$.  In statements where we only consider the image of the gate map in the hyperbolic spaces, we simplify notation and write $\mathfrak g_{\mathbf F_U}$.
\end{rem}

The following lemma  provides a formula for computing the 
distance between a point and a product region.  It  
is an immediate consequence of \cite[Corollary~1.28]{BehrstockHagenSisto:quasiflats}; we give a sketch of the proof here for completeness.
\begin{lem} 
    [\cite{BehrstockHagenSisto:quasiflats}]\label{lem:disttoprod}
Let $(\X,\s)$ be a hierarchically hyperbolic space. Fix $U\in\s$ and 
let 
 $\mathcal Y=\{Y\in\s \mid Y\trans U \textrm{ or } Y\sqsupsetneq U\}$.
Then for all $s\geq s_0$ and any $x\in\X$,
\begin{align} \label{eqn:disttoprod}
\dist_{\X}(x,\mathbf P_U)\asymp_{C,D} \sum_{Y\in \mathcal 
Y}\ignore{\dist_Y(x,\rho^U_Y)}{s},
\end{align} where $s_0,C,$ and $D$ are the constants from Theorem \ref{thm:distance_formula}.
\end{lem}

\begin{proof}[Sketch of proof]
To each bounded set $\mathcal A\subset\X$, we associate a tuple $(\mathcal A_V)_{V\in\s}$ whose components are the projections of $\mathcal A$ to $\fontact V$ for each $V\in \s$, i.e., $\mathcal A_V=\pi_V(\mathcal A)$. We will consider the case $\mathcal A=\mathfrak g_{\mathbf P_U}(x)\subset \mathbf P_U$.  By \cite[Remark~1.16]{BehrstockHagenSisto:asdim}, if $V\nest U$ or $V\perp U$, we have $\pi_V(\frak g_{\mathbf{P}_U}(x))=\pi_V(x)$.  Combining this with the definition of $\mathbf{P}_U$ (Definition \ref{def:productregion}), we have 
\[
(\mathfrak g_{\mathbf P_U}(x))_V=\begin{cases} \rho^U_V & \textrm{if $V\in\mathcal Y$} \\ \pi_V(x) & \textrm{otherwise}.\end{cases}
\]

There is a constant $K_0$ depending only on $(G,\s)$ such that $\dist_{\X}(x,\mathbf P_U)\asymp_{K_o}\dist_{\X}(x,\mathfrak g_{\mathbf P_U}(x))$ by \cite[Lemma~1.27]{BehrstockHagenSisto:quasiflats}.  From the above discussion we see that the only components of the tuple $(x_V)_{V\in\s}$ associated to $x$ and the tuple $((\mathfrak g_{\mathbf P_U}(x)_V)_{V\in\s}$ associated to $\mathfrak g_{\mathbf P_U}(x)$ which differ in $\fontact V$ occur when $V\in\mathcal Y$.  Thus the distance from $x$ to $\mathfrak g_{\mathbf P_U}(x)$ in $\X$ can be approximated using only the domains $V\in\mathcal Y$.
\end{proof}

Lemma~\ref{lem:disttoprod} gives  the following geometric picture.  Let $x,y\in \X$ and $U\in \s$, and consider $x'=\frak g_{\mathbf P_U}(x)$ and $y'=\frak g_{\mathbf P_U}(y)$.  Let $V$ be a domain that is relevant for $x$ and $y$.  Then any distance in $\fontact V$  contributes either to the distance from $x$ or $y$ \emph{to} the product region $\mathbf P_U$ or to the distance \emph{within} $\mathbf P_U$, but not both (see Figure \ref{fig:Lemmadisttoprod}).  In particular, if $V\trans U$ or $V\sqsupsetneq U$, then $V$ is relevant for either $x,x'$ or $y,y'$ but not for $x',y'$.  Any other $V$ is relevant for $x',y'$ but not for $x,x'$ or $y,y'$.  
\begin{figure}
\centering
 \def\svgwidth{2.5in}
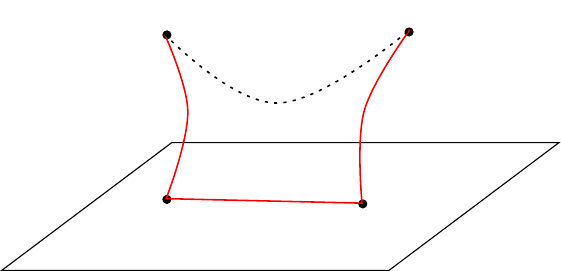
\caption{Geometric picture of Lemma \ref{lem:disttoprod}.  Domains which are relevant for $x,y$ are relevant for either the horizontal red segment or (at least one of) the vertical red segments. }
\label{fig:Lemmadisttoprod}
\end{figure}

\subsection{Axial elements in hierarchically hyperbolic groups}

	Let $(G,\s)$ be a hierarchically hyperbolic group, and fix the
	constant $L_S$ from Lemma~\ref{lem:ellnerve} and
	$\ell\in[0,\delta]$.  (Note that the constant $\delta$ is part of
	the definition of $(G,\s)$; see Definition \ref{defn:HHS}.)
	Following
    \cite{DurhamHagenSisto:HHS_boundary}, for an element $h\in G$ we define
    \begin{equation}\label{def:bigset}\B(h)=\{U\in \s\mid \pi_U(\langle h\rangle) \textrm{ is
    unbounded}\}.\end{equation}

\begin{lem}\label{lem:finiteorder}
Let $(G,\s)$ be a hierarchically hyperbolic group.  An element $h\in
G$ is finite order if and only if $\B(h)=\emptyset$.
\end{lem}

\begin{proof} In \cite[Proposition~6.4]{DurhamHagenSisto:HHS_boundary}
it is proven that an element  $h\in G$ is elliptic if and only
if $\B(h)=\emptyset$.  The result follows from this, since a group element acts elliptically on its Cayley graph if and only
if the element is of finite order.
\end{proof}

    \begin{defn}[Axial element]\label{def:axial}
    An element $h\in G$ with $\B(h)\neq\emptyset$ is called \emph{axial}.
    \end{defn}
   
     Lemma~\ref{lem:finiteorder} shows that every infinite order element of a hierarchically hyperbolic group is axial.
By \cite[Lemma~6.7]{DurhamHagenSisto:HHS_boundary},  the elements of $\B(h)$ are pairwise
orthogonal.  As the number of pairwise orthogonal domains in a hierarchically hyperbolic group is uniformly bounded by the constants in the definition of a hierarchically hyperbolic space \cite[Lemma~2.1]{BehrstockHagenSisto:HHS2}, it  follows  that $|\B(h)|$ is uniformly bounded independently of the choice of $h$.  
As noted in \cite{DurhamHagenSisto:HHS_boundary}, since $h\colon\fontact U\to\fontact(hU)$ is an isometry, we have $h U\in \B(h)$ whenever $U\in \B(h)$.  Moreover, by \cite[Lemma~6.3]{DurhamHagenSisto:HHS_boundary}, there is a constant $M$ depending only on the constants in the definition of a hierarchically hyperbolic space such that for all $h\in G$ and $U\in\B(h)$, we have $h^M  U=U$.  In other words, by passing to a uniform power, we may assume that $h$ fixes its big set elementwise.   
Moreover, by passing to this uniform power, we may assume that $h$ is a loxodromic isometry of $\fontact U$ for any $U\in\B(h)$ by \cite[Theorem~3.1]{DurhamHagenSisto:HHS_boundary_symphoria}.  We let $\tau_U(h)$
denote the stable translation length of $h$ in this action and let $\gamma_h^U$ be a $(2,\ell)$--quasi-geodesic axis of $h$ in $\fontact U$ (see Definition \ref{def:qgeoaxis}).

\begin{rem}[{Acylindrical actions}]\label{rem:acylindricity} The action of a group $G$ on a metric space $X$ is \emph{acylindrical}
if for all $\varepsilon\geq 0$, there exist constants $R(\varepsilon),N(\varepsilon)\geq 0$ such
that for all $x,y\in X$ satisfying $\dist_X(x,y)\geq R(\varepsilon)$, there are at most $N(\varepsilon)$ elements $g\in G$ for which 
$\dist_X(x,gx)\leq \varepsilon$ and $\dist_X(y,gy)\leq
\varepsilon$. By \cite[Theorem~K]{BehrstockHagenSisto:HHS1}, $G$ acts acylindrically on $\fontact S$, where $S$ is the $\nest$--maximal element of $\s$. An immediate consequence of this is a lower bound on the translation length $\tau_S(h)$ that depends only on the hierarchy constants \cite[Lemma~2.2]{Bowditch:tight}.

Let $U\in\s$, and let $H$ be a subgroup of $G$ which fixes $U$, so that $H$ acts on $\fontact U$.  
 If $U\neq S$, it is not necessarily the case that $H$ acts acylindrically on $\fontact U$, and it remains an open question whether there is a uniform lower bound on $\tau_U(h)$ in general.  We deal with this issue in the present paper by assuming such a uniform lower bound as a hypothesis.  \textit{Hierarchical acylindricity} is a standard assumption requiring that the action of $H$ on $\fontact U$ is acylindrical for all such $U$: this would also ensure a uniform lower bound on translation length.
\end{rem}

The next lemma is a straightforward consequence of the hyperbolicity of the spaces $\fontact U$.

\begin{lem} \label{axiallem}
Let $(\cuco X,\s)$ be a hierarchically hyperbolic space, and let $G$ be
a group acting geometrically on $\cuco X$.  Fix a basepoint $x_0\in
\cuco X$, the constant $L_S$ from Lemma~\ref{lem:ellnerve}, and
$\ell\in[0,\delta]$.  Then there exist constants $K_0,L\geq 0$ such
that the following holds.  Let $h\in G$ be an axial element so that $hU=U$ for each $U\in \B(U)$  and $\tau_U(h)\geq L_S\delta$.  
For any
$k\geq K_0$   let $x',y'\in\fontact
U$ be the closest points on $\gamma^U_h$ to $\pi_U(x_0)$ and $\pi_U(h^kx_0)$,
respectively.

There exists a point $\xi$ on the subpath of $\gamma_h^U$ from $x'$ to $y'$ so that $\dist_U(\xi,x')\leq L$ and $\dist_U(\xi,\pi_U(\mu(x,y)))\leq L$. 
\end{lem}

\begin{proof}

Recall that the image of any $(\lambda,\lambda)$--hierarchy path in $\fontact U$ is a (unparametrized) $(\lambda,\lambda)$--quasi-geodesic.  Since the axis of $h$ in $\fontact U$ is a $(2,\ell)$--quasi-geodesic, the concatenation (in the appropriate order) of $\pi_U(\mu(x_0,h^kx_0))$, $[\pi_U(x_0),x']_{\fontact U}$, $[\pi_U(h^kx_0),y']_{\fontact U}$, and a subpath of $\gamma_h^U$ forms a $(2,\ell)$--quasi-geodesic quadrilateral $Q$ in $\fontact U$.  Let $M$ be the Morse constant associated to $(2,\ell)$--quasi-geodesics in a $\delta$--hyperbolic space.   Fix $K_0$ so that $L_S\delta K_0>4\delta+4M+1$, and let $k\geq K_0$.  Note that $K_0$ is independent of the choice of axial element $h$.

The quadrilateral $Q$ is  $(2M+2\delta)$--thin, that is, given any point $z$ on a side of $Q$, there is a point on one of the other three sides of $Q$ at distance at most $2\delta+2M$ from $z$. 
Let $v$ and $w$ be points on the subpath of $\gamma_h^U$ between $x'$ and $y'$ so that $\dist_S(x',v)=\lceil 4\delta+4M+1\rceil$ and $\dist_S(y',w)=\lceil 4\delta+4M+1\rceil$.  We claim that the subpath $\beta$ of $\gamma_h^U$ from $v$ to $w$ is contained in the $(2\delta+2M)$--neighborhood of $\pi_U(\mu(x,y))$.  Let $z$ be a point on the subpath of $\gamma_h^U$ from $x'$ to $y'$.  Then there is a point $z'$ on one of the other three sides satisfying $\dist_U(z,z')\leq 2M+2\delta$. Suppose  $z'$ lies on the geodesic $[\pi_U(x),x']$.  As $x'$ is the nearest point on $\gamma_h^U$ to $x$ (hence also to $z'$), we must have $\dist_U(z',x')\leq \dist_U(z',z)\leq 2M+2\delta$.  The same holds if $z'$ lies on the geodesic $[\pi_U(y),y']$.  Thus if $z$ lies on $\beta$, then $z'$ must lie on $\pi_U(\mu(x,y))$, as desired.

Since the map $\pi_U$ is $G$--equivariant and, in particular,  $h^k\pi_U(x_0)=\pi_U(h^kx_0)$, we also have $y'=h^kx'$.  Thus $\dist_U(x',y')\geq k\tau_U(h)\geq K_0 L_S\delta \geq 4\delta+4M+1$.  It follows that $\beta$ is non-empty.  We let $\xi$ be the point on $\beta$ closest to $x'$, so that $\dist_U(\xi,x')= \lceil 4\delta + 2M + 1\rceil$.  Taking $L=\lceil 4\delta + 2M + 1\rceil$ completes the proof.
\end{proof}

\section{Proof of Theorem~\ref{mainthm}}
Let $(G,\s)$ be a hierarchically hyperbolic group.  The authors and
Durham show in \cite[Corollary~3.8]{AbbottBehrstock:HHSlargest} that 
by possibly changing the hierarchy structure on $G$, we may assume
that $(G,\s)$ has \emph{unbounded products}.  In this paper, we don't 
directly use the definition of unbounded products, rather we only need the
following consequence about Morse elements in the structure $(G,\s)$, 
which follows from \cite[Theorem~4.4 \&
Corollary~5.5]{AbbottBehrstock:HHSlargest}: if $h\in G$ is an infinite
order Morse element, then $h$ is axial and $\B(h)=\{S\}$, where $S$ is the $\nest$--maximal element of $\s$.

 We begin by fixing the constants that will be used throughout the
 proof.  Definition \ref{defn:HHS} provides a constant $\delta$ such
 that $\fontact U$ is $\delta$--hyperbolic for all $U\in\s$.  Let
 $L_S$ be the constant from Lemma~\ref{lem:ellnerve}, and fix
 $\ell\in[0,\delta]$.  Let $E$ be as in Notation \ref{not:E}; in 
 particular, $E$ is larger than any of the hierarchy constants for
 $G$.  Let $T$ be the lower bound on translation length in the acylindrical action on $\fontact S$ noted in Remark~\ref{rem:acylindricity}. 
Fix $\lambda\geq \max\{2,\ell\}$ so that any two points $x,y\in G$ are 
connected by a  $(\lambda,\lambda)$--hierarchy path. 
Let $K_0,L$ be the constants from Lemma \ref{axiallem}, and fix a constant $R> 2E$.

Finally, set 
\begin{equation}\label{eqn:Kdef}
K= \max\{2\delta, R, K_0, \left\lceil\frac{4L+3E}{T}\right\rceil+2, 2L\}.
\end{equation}
This constant $K$ is uniform, in the sense that it depends only on the hierarchy constants for $(G,\s)$.

Let $a,b\in G$ be two infinite order Morse elements and suppose there
exists $g\in G$ such that $ga=bg$.  Since $(G,\s)$ has unbounded
products, we have $\B(a)=\B(b)=\{S\}$.  For simplicity of notation, we denote the asymptotic
translation length of $b$ in $\fontact S$ by $\tau(b)$.  Note that $S$
is fixed by the action of $G$ on $\s$.  Since $g$ conjugates $a^i$ to
$b^i$ for any $i\in \mathbb Z$, we first replace $a$ and $b$ by
sufficiently high powers so that $\tau(b)\geq L_S\delta$.  
By
Remark~\ref{rem:acylindricity}, such a
power can be chosen uniformly (that is, depending only on the
hierarchy constants, and not the choice of elements $a$ and $b$).

Let $\gamma_b=\gamma^S_b$ be a $(2,\ell)$--quasi-geodesic axis of
  $b$ in $\fontact S$.  Then $\gamma_a=\gamma^S_a =g^{-1}\gamma_b$ is a $(2,\ell)$--quasi-geodesic axis of $a$ in $\fontact S$.  
  We now fix a quadrilateral of $(\lambda,\lambda)$--hierarchy paths in $G$:
  $\mu(1,g)$, $\mu(1,b^K)$, $b^K\mu(1,g)=\mu(b^K,b^Kg)$, and
  $g\mu(1,b^K)=\mu(g,gb^K) =\mu(g,a^Kg)$.

Our first step is to replace $g$ with a different conjugator whose
length we are able to bound in $G$.  Since $K\geq K_0$, we may apply
Lemma \ref{axiallem} to each of the axes $\gamma_a$, $\gamma_b$ in
$\fontact S$ and the points $1,a^K\in G$ and $1,b^K\in G$,
respectively.  This yields a point $z'\in \gamma_a$ and a point $w'\in
\gamma_b$ such that $z'\in \mathcal N_{L}(\pi_S(\mu(1,a^K)))$ and
$w'\in \mathcal N_{L}(\pi_S(\mu(1,b^K)))$.  Moreover, if $x$ is a
point on $\gamma_a$ nearest to $\pi_S(1)$ and $y$ is a point
on $\gamma_b$ nearest to $\pi_S(1)$, then $\dist_S(z',x)\leq L$ and
$\dist_S(w',y)\leq L$.  See Figure \ref{fig:Case1}.  Let $z\in
\pi_S(\mu(1,a^K))$ and $w\in\pi_S(\mu(1,b^K))$ be points nearest 
to $z'$ and $w'$, respectively.  Since $g$ fixes $S$, we have $gz\in
\pi_S(\mu(g,ga^K))=\pi_S(\mu(g,b^Kg))\subseteq \fontact S$ and $gz'\in
g\gamma_a=\gamma_b$.  Since $g$ is an isometry, we have
$\dist_S(gz',gx)=\dist_S(z',x)\leq L$.

\begin{figure}[h]
  \centering 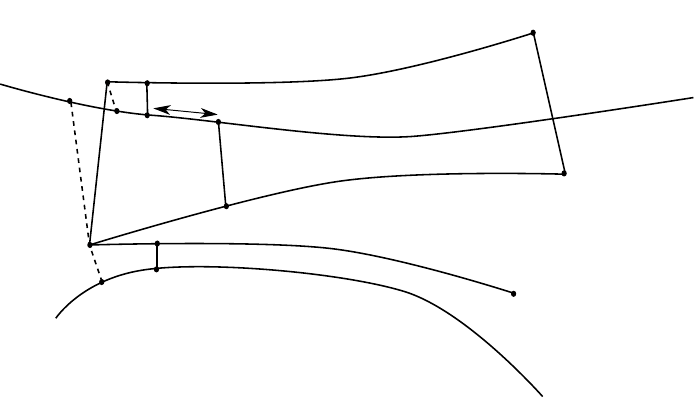	
  \caption{The geometry of the axes of $a$ and $b$ in $\fontact S$.} 
	\label{fig:Case1}
\end{figure}

By possibly premultiplying $g$ by a power of $b$, we may assume that
$\dist_S(gz',w')\leq \tau(b)$ (while still conjugating $a$
  to $b$).  Thus we have
\begin{equation}\label{eqn:taub}
\dist_S(y,gx)\leq \dist_S(y,w')+\dist_S(w',gz')+\dist_S(gz',gx)\leq \tau(b)+2L.
\end{equation}

Our goal is to bound the length of this new conjugator, which by an abuse of notation we will still call $g$.

We will show that for each $U\in\s$, we have 
\begin{equation}\label{eqn:Ubound}
\dist_U(1,g)\leq
2K\dist_U(1,b)+\dist_U(g,b^Kg)+K,
\end{equation} where $K$ is as in \eqref{eqn:Kdef}.  After establishing this bound for
each $U\in \s$, we then apply the distance formula (Theorem \ref{thm:distance_formula}) with threshold $R$ to obtain 
\[
\dist_G(1,g)\preceq_{C,D} 2K\dist_G(1,b)+\dist_G(g,b^Kg)+K.
\]
Finally, we use the fact that $\dist_G(g,b^{K}g)=\dist_G(g,g
a^{K})=\dist_G(1,a^{K})\leq K\dist_G(1,a)$, which establishes that
\begin{equation} \label{eq:case2}
\dist_G(1,g)\preceq_{C,D} 2K\dist_G(1,b)+K\dist_G(1,a)+K,
\end{equation}
where $C,D$ are the constants given by the distance formula (Theorem
\ref{thm:distance_formula}).  (Note that by assumption, $R$ is 
sufficiently large to serve as a threshold in the distance formula.)  
This will provide the desired bound in $G$.

\medskip

Fix $U\in\s$. If $U\not\in \relevant(1,g;R)$, then we have $\dist_U(1,g)\leq R\leq K$, and \eqref{eqn:Ubound} holds.  Thus we assume for the rest of the proof that $U\in \relevant(1,g;R)$.   There are two cases to consider: either $U=S$ or $U\propnest S$.  We will deal with each of these possibilities individually.

\subsubsection*{{\bf Case 1: $U=S$}}
In this case we have (as seen in Figure
\ref{fig:Case1}): 
\begin{align*}
\dist_S(1,g)&\leq \dist_S(1,w)+\dist_S(w,w')+\dist_S(w',gz')+\dist_S(gz',gz)+\dist_S(gz,g) \\
& \leq \dist_S(1,b^K)+2L+\tau_S(b)+\dist_S(g,b^Kg) \\
& \leq 2K\dist_S(1,b)+\dist_S(g,b^Kg)+K,
\end{align*}
where the final inequality follows from the fact that
$\dist_S(1,b)\geq \tau_S(b)$ and \eqref{eqn:Kdef}.  Therefore \eqref{eqn:Ubound} holds in this case.

\subsubsection*{{\bf Case 2: $U\propnest S$}}
 As we are assuming that $U$ is $R$--relevant for $1,g$, we must have
 $\rho^U_S\subseteq \fontact S$ is contained in the $E$--neighborhood
 of a geodesic in $\fontact S$ from $\pi_S(1)$ to $\pi_S(g)$ by the bounded geodesic image axiom (Definition \ref{defn:HHS}(\ref{item:dfs:bounded_geodesic_image})).  As
 geodesic quadrilaterals in $\delta$--hyperbolic spaces are $2\delta$--thin, 
 it follows that $\rho^U_S$ is contained in the
 $(E+2\delta)$--neighborhood of
 $[\pi_S(1),y]\cup[y,gx]\cup[gx,\pi_S(g)]$, where these are geodesics in $\fontact S$.  Since $x$ and
 $gx$ are the nearest point projections of $\pi_S(1)$ and $\pi_S(g)$ onto
 $\gamma_b$, respectively, it follows from \eqref{eqn:taub} that the
 projection of $[\pi_S(1),y]\cup[y,gx]\cup[gx,\pi_S(g)]$ onto
 $\gamma_b$ has diameter at most $\tau(b)+2L$.  In particular,
 since nearest point projection maps in hyperbolic spaces are Lipschitz, the nearest point on $\gamma_b$ to $\rho^U_S$ is distance at most
 $\tau(b)+2L+E$ from $y$.

By an analogous argument, if $U$ is also $R$--relevant for $b^K,b^Kg$, we must have that $\rho^U_S$ is contained in the $(E+2\delta)$--neighborhood of $[\pi_S(b^K),b^ky]\cup[b^ky,b^kgx]\cup[b^kgx,\pi_S(b^kg)]$.  In particular, the nearest point on $\gamma_b$ to $\rho^U_S$ is at distance at most $\tau(b)+2L+E$ from $b^Ky$.

However, our choice of $K$ in \eqref{eqn:Kdef} ensures that 
\[
\dist_S(y,b^Ky)\geq K\tau(b)\geq \left(\left\lceil\frac{4L+3E}{T}\right\rceil+2\right) \tau(b)\geq 4L+3E+2\tau(b),
\] 
which is a contradiction.  Therefore $U$ is not $R$--relevant for $b^K,b^Kg$, and so $\dist_U(b^K,b^Kg)\leq R$.

Therefore,
\begin{align*}
\dist_U(1,g)&\leq \dist_U(1,b^K)+\dist_U(b^K,b^Kg)+\dist_U(b^Kg,g) \\
&\leq K\dist_U(1,b)+\dist_U(b^K,b^Kg)+R \\
& \leq K\dist_U(1,b)+\dist_U(b^K,b^Kg)+K,
\end{align*}
where the final inequality follows from our choice of $K$ in \eqref{eqn:Kdef}.

Therefore \eqref{eqn:Ubound} holds in this case, which completes the  proof of theorem. \qed

\section{A family of hierarchically hyperbolic groups}\label{sec:shhg}

In this section, we highlight three properties which isolate some of
the nice features of compact 
special groups and which appears in many other contexts as well.  We will show in Proposition
\ref{prop:specialHHGs} that many hierarchically hyperbolic groups
satisfy these three properties, which we call \emph{$\mathbf F_U$
stabilizers}, \emph{orthogonal decomposition}, and
\emph{commutativity}.

Fix a hierarchically hyperbolic group $(G,\s)$.  
If $\mc U\subseteq
\s$ is a collection of pairwise orthogonal domains, we denote the
\emph{container} of $\mc U$ in $S$ by $C_\mc U$ (Definition
\ref{defn:HHS}(\ref{item:dfs_orthogonal})); by definition, each domain
$V$ which is orthogonal to every $U\in \mc U$ is nested into $C_\mc
U$.  We say $G$ has \emph{clean containers} if for every 
collection of pairwise orthogonal domains $\mathcal U$, the container $C_\mc U$ is orthogonal to
every $U\in \mc U$.  If $\mc U=\{U\}$, we write $C_U$ instead of
$C_{\{U\}}$.

Recall that for any domain $U\in\s$, we identify $\mathbf P_U$ with $\mathbf F_U\times\mathbf E_U$ (see the discussion after Definition \ref{defn:orthogonal_partial_tuple}).  If a subgroup $H\leq G$ fixes a domain $U\in\s$ (in the action of $G$ on $\s$), then whenever $V\nest U$ or $V\perp U$, we have $hV\nest U$ or $hV \perp U$, respectively, for each $h\in H$.  It follows that $H$  stabilizes the product region $\mathbf P_U$ and each of its factors $\mathbf F_U$ and $\mathbf E_U$.  
\begin{defn}\label{def:G_U}
For any $U\in \s$, let $G_U$ be the subgroup of $G$ that fixes $U$ in
the action of $G$ on $\s$ and that stabilizes $\mathbf F_U\times\{e\}$
for each $e\in\mathbf E_U$.
\end{defn} 
Equivalently, $G_U$ is the subgroup which stabilizes each factor of $\mathbf F_U\times\mathbf E_U$ and acts as the identity on the second factor.  We note that when $G$ has clean containers, the second factor $\mathbf E_U$ is isometric to $\mathbf F_{C_U}$ by Lemma \ref{lem:uniquecleancontainers}.

\medskip

\begin{exmp}\label{ex:specialcommute}
	Right-angled Artin groups and, more generally, compact special
	groups, 
	provide a good example to have in mind when reading this section.
	With the standard hierarchically hyperbolic group structure given in \cite{BehrstockHagenSisto:HHS1}, such
	groups are hierarchically acylindrical, and have clean
	containers \cite[Proposition~7.2]{AbbottBehrstock:HHSlargest}.
	One nice property of right-angled Artin groups is that two
	elements commute if and only if all the generators in a cyclically
	reduced factorization of one of the elements commute with all the
	generators in a cyclically reduced factorization of the other
	element.  Hence, in the Salvetti complex of a right-angled Artin
	group, $G$, we have that two elements span a periodic plane if and only
	if they commute.  Similarly, if a group is compact special it embeds as a quasi-convex subgroup of a right-angled
	Artin group and thus inherits this property as well.  Further, if
	a group is virtually compact special, then, up to taking powers,
	two elements commute if and only if they span a periodic plane.
	For these groups $U,V\in\frakS$ are orthogonal if and only if they
	have associated subcomplexes of the cube complex which span a
	direct product.  Hence, it follows that given $U,V\in\frakS$ which
	are orthogonal, the subgroup which fixes $U$ in the action on
	$\frakS$ and which stabilizes the subset $\mathbf F_U \times\{e\}$
	for each $e\in\mathbf E_U$ has the property that it commutes with
	the similarly defined subset for $V$.  In other words, elements of
	$G_U$ and $G_V$ commute. In particular, if $g\in G$ fixes each $U_{i}\in\B(G)$ then 
	$g$ can be 
	written as a product of elements in $G_{U_{i}}$.
\end{exmp}

\subsection{The $\mathbf F_U$ stabilizers, orthogonal decomposition, and commutative properties}
We will now extract and formalize the properties which we described above for 
right-angled Artin groups.

\subsubsection{The $\mathbf F_U$ stabilizers property}
Since $(G,\s)$ is a hierarchically hyperbolic group, there is a finite fundamental  domain $\s'$ for the action of $G$ on $\s$.  We may choose $\s'$ to have the property that for each $U\in \s'$, there exists $e\in \mathbf E_U$ such that $1\in \mathbf F_U\times \{e\}$, where $1$ is the identity element of $G$.  We denote this copy of $\mathbf F_U$ by $\FixedF_U$.  For such domains $U$, we always have $G_U\subseteq \FixedF_U$.  To see this, consider any  $f\not\in \FixedF_U$. Since $1\in \FixedF_U$ and $f=f\cdot 1\not\in \FixedF_U$, the element $f$ does not stabilize $\FixedF_U$, so $f\not\in G_U$.  

The first property says that for all $U\in \s'$, the sets $\FixedF_U$ and $G_U$ are coarsely equal.

\begin{defn}[$\mathbf F_U$ stabilizers]\label{def:FUStabilizers}
A hierarchically hyperbolic group $(G,\s)$ satisfies the
\emph{$\mathbf F_U$ stabilizers property} if there exists a constant
$\nu$ depending only on the hierarchy constants such that $\dist_G(f,G_U)\leq \nu$ for each $U\in
\s'$ and any $f\in \FixedF_U$.
\end{defn}

The $\mathbf F_U$ stabilizers property implies that for domains $U\in
\s'$, the subgroup $G_U$ inherits many geometric properties 
from 
$\FixedF_U$, including hierarchical quasiconvexity.  In a
hierarchically hyperbolic group, there is a function $k\colon
[0,\infty)\to [0,\infty)$ so that for any $U\in \s$, the subspace
$\mathbf F_U$ is $k$--hierarchically quasiconvex
\cite[Construction~5.10]{BehrstockHagenSisto:HHS2}. If the group has
the $\mathbf F_U$ stabilizers property, then since $G_U$ and $\FixedF_U$ are at uniformly bounded distance whenever $U\in \s'$, there is a
function $k'\colon [0,\infty)\to [0,\infty)$ depending only on $k$ and
$E$ so that the subgroup $G_U$ is also hierarchically quasiconvex for
any $U\in\s'$.  It then follows from
\cite[Lemma~5.5]{BehrstockHagenSisto:HHS2} that there is a
well-defined gate map $\mathfrak g_{G_U}\colon G\to G_U$.  Moreover,
for any $g\in G$, each coset $gG_U$ of $G_U$ is also
$k'$--hierarchically quasiconvex in $G$, so we also have a
well-defined gate map $\mathfrak g_{gG_U}\colon G\to gG_U$.  These
gate maps will be important for defining the two additional properties
we introduce in this section.

The next lemma says that $(G_U,\s_U)$ is a hierarchically hyperbolic group, where
$\s_U=\{V\in \s\mid V\nest U\}$.  
\begin{lem}
Let $(G,\s)$ be a  hierarchically hyperbolic group satisfying the $\mathbf F_U$ stabilizers property.  For any $U\in \s'$, $(G_U,\s_U)$ is a hierarchically hyperbolic group.
\end{lem}

\begin{proof}
The $\mathbf F_U$ stabilizers property says that $G_U$ is at uniformly bounded distance from $\FixedF_U$.  In particular, $\FixedF_U$ and $G_U$ are quasi-isometric.  Since $(\FixedF_U,\s_U)$ is a hierarchically hyperbolic space \cite[Proposition~5.11]{BehrstockHagenSisto:HHS2}, this immediately implies that $(G_U,\s_U)$ is a hierarchically hyperbolic space, where the associated hyperbolic spaces and maps are the same as those for $(\FixedF_U,\s_U)$. It remains to show that $(G_U,\s_U)$ is a hierarchically hyperbolic
\emph{group}.  For this, note that $G_U$ stabilizes $\s_U$ by
definition.  Since $(G,\s)$ is a hierarchically hyperbolic group and
$G_U\leq G$, the four additional conditions from Definition
\ref{defn:HHG} hold because they hold for the action of $G$ on $\s$.  For
example, since $G$ acts cofinitely on $\s$ and preserves the relations
$\nest, \trans$, and $\perp$, so does $G_U$.  Similar arguments show the other
three conditions hold.
\end{proof}

\subsubsection{The orthogonal decomposition property}
The next property allows any infinite order element which fixes a collection of 
pairwise-orthogonal domains 
$\{U_1,\dots, U_k\}$ to be decomposed into a product of elements in
$G_{U_i}$.  Before defining this property, the following lemma 
establishes that for each $i=1,\dots, k$ there is a preferred 
$\mathbf F_{U_i}\times
\{e_i\}$ which we denote  by $\FixedF_{U_i}$. The careful 
reader will note that if $U_i$ is already
in the fundamental domain $\s'$, then the choice given by the lemma 
is consistent with our previous choice of $\FixedF_{U_i}$.

\begin{lem}\label{lem:choiceofFU}
Let $(G,\s)$ be a hierarchically hyperbolic group with the $\mathbf
F_U$ stabilizers property, and let $\mathcal U=\{U_1,\dots, U_k\}$ be
a maximal collection of pairwise-orthogonal domains in $\s$.  Then
there exist $t\in G$ and copies $\mathbf F_{U_i}\times\{e_i\}$ such
that the following hold for all $i$:
\begin{itemize}
\item $U_i=tU_i'$ for some $U_i'\in \s'$;
\item $\mathbf F_{U_i}\times \{e_i\}=t\FixedF_{U_i'}$;
\item $G_{U_i}=t G_{U_i'} t^{-1}$; and 
\item $\dist_G(t,\frak g_{\mathbf P_{\mc U}}(1))\leq E\nu$, where $\nu$ is the constant from Definition \ref{def:FUStabilizers}.
\end{itemize}
\end{lem}

\begin{proof} Consider the product region $\mathbf P_{\mc U}$ associated to $\mathcal U$, and let $t'$ be any point in $\frak g_{\mathbf P_{\mc U}}(1)$.  For the first part of the proof, it will be convenient to distinguish between the abstract product region  $\mathbf P_{\mc U}=\mathbf F_{U_1}\times \cdots \times \mathbf F_{U_k}$ and its image $\phi_{\mc U}(\mathbf P_{\mc U})\subseteq G$ (see the discussion after Definition \ref{defn:orthogonal_partial_tuple}).  Let $(t_1',\dots, t_k')\in \mathbf P_{\mc U}$ be such that $\phi_{\mc U}(t_1',\dots, t_k')=t'$.    We will adjust each $t_i'$ individually to find a new point $(t_1,\dots, t_k)$, which will determine the points $e_i$ in the statement.  At the $i$th stage, we change the $i$th coordinate of the point in $\mathbf P_{\mc U}$ to ensure that it lies in a coset of $G_{U_i}$ that is completely contained in the associated copy of $\mathbf F_{U_i}$.  In subsequent steps, we will adjust later coordinates: this may change which coset of $G_{U_i}$ the point lies in, but it will simultaneously translate the copy of $\mathbf F_{U_i}$ so that this new coset is still contained in the new copy of $\mathbf F_{U_i}$, as desired.  After changing all coordinates, the desired element $t$ will be $\phi_{\mc U}(t_1,\dots, t_k)$.    

We begin with $i=1$.  Since $\s'$ is a fundamental domain, 
there is some $f_1'\in G$ and $U_1'\in \s'$ such that $\phi_{\mc
U}(\mathbf F_{U_1},t_2',\dots, t_k')=f_1'\FixedF_{U_1'}$.  Since $G_{U_1'}\subseteq \FixedF_{U_1'}$, we
have $f_1'G_{U_1'}\subseteq \phi_{\mc U}(\mathbf
F_{U_1},t'_2,\dots, t'_k)$.  By the $\mathbf F_U$ stabilizers
property, there is an element $t_1\in \mathbf F_{U_1}$ with
$\dist_G(t',\phi_{\mc U}(t_1, t_2',\dots, t_k'))\leq \nu$ and $\phi_{\mc
U}(t_1,t_2',\dots, t_k')\in f_1'G_{U_1'}$.  

We fix $t_1$ from the previous paragraph and now consider $i=2$.  The
point $\phi_{\mc U}(t_1,t_2',\dots, t_k')$ is in $\phi_{\mc
U}(t_{1},\mathbf F_{U_2},t_3',\dots, t_k')$.  Again, as above, there is some $f_2'\in
G$ and $U_2'\in \s'$ for which $\phi_{\mc U}(t_{1},\mathbf F_{U_2},t_3',\dots,
t_k')=f_2'\FixedF_{U_2'}$.  Also, as above, we can find an
element $t_2\in \mathbf F_{U_2}$ with $\dist_G(\phi_{\mc
U}(t_1,t_2',\dots, t_k'),\phi_{\mc U}(t_1, t_2,t_3',\dots, t_k'))\leq
\nu$ and $\phi_{\mc U}(t_1,t_2,t_3',\dots, t_k')\in f_2'G_{U_2'}$.

Continuing in this way for each $i$ yields a point $(t_1,t_2,\dots, t_k)\in \mathbf F_{U_1}\times\cdots\times \mathbf F_{U_k}$.  Letting $t=\phi_{\mc U}(t_1,\dots, t_k)$, it follows from the triangle inequality that  
\begin{equation}\label{eqn:disttt'}
\dist_G(t',t)\leq k\nu\leq E\nu,
\end{equation} where the final inequality holds because any collection of pairwise orthogonal domains has cardinality bounded by $E$.   

We now return to our convention of identifying $\mathbf P_{\mc U}$ with its image $\phi_{\mc U}(\mathbf P_{\mc U})\subseteq G$.  We have shown that, for each $i$, we have $t\in \mathbf F_{U_i}\times \{e_i\}$ for some $e_i$.  Precisely, $e_i=\phi_U^\perp(t_1,\dots, \hat t_i,\dots, t_k)$, where $\hat t_i$ indicates that the term $t_i$ does not appear in the tuple.  

We now show that  $\mathbf F_{U_i}\times \{e_i\}$ satisfies the conclusion of the lemma for each $i$.  The final bullet point holds by \eqref{eqn:disttt'}.

There is an element $f_i\in G$ such that $\mathbf F_{U_i}\times \{e_i\}=f_i\FixedF_{U_i'}$, where $U_i=f_iU_{i}'$, and $t\in f_iG_{U_i}$.  
  Thus $f_iG_{U_i'}=t G_{U_i'}$, and $G_{U_i}=t G_{U_i} t^{-1}$, so the third bullet point holds.  Also, since $t=f_iq_i$ for some $q_i\in G_{U_i'}$,  we have 
  \[
  t\FixedF_{U_i'}=f_iq_i\FixedF_{U_i'}=f_i\FixedF_{U_i'},
  \]
  so the second bullet point holds.  Finally, $tU_{i}'=f_iq_iU_i'=f_iU_i'=U_i$, which shows that the first bullet point holds and concludes the proof of the lemma.
\end{proof}

The following lemma is presumably well-known, but is 
not in the literature.
An immediate corollary of this is that an 
 axial element fixes the container associated to its big set.

\begin{lem}\label{lem:uniquecleancontainers}
Let $(G,\frakS)$ be a hierarchically hyperbolic group with clean containers, and let  $\{U_1,...,U_k\}$ be a (non-maximal) collection of pairwise orthogonal domains.  There exists a unique $C\in\frakS$ such that: if for each $i$,  a domain $V\in\frakS$ satisfies $V\perp U_i$ then $V\nest C$.
\end{lem}

\begin{proof}
First, by Definition \ref{defn:HHS}(\ref{item:dfs_orthogonal}) some $C$ exists with the desired property, what is needed is to prove uniqueness. So suppose that both $C$ and $C'$ satisfy this property. Since the containers are clean, each of $C$ and $C'$ are orthogonal to $U_i$ for each $i$. Thus, since $C$ is a container and since $C'$ is orthogonal to all the $U_i$, we must have that $C'\nest C$. Similarly, $C\nest C'$. Thus $C=C'$, as desired.
\end{proof}

\begin{defn}[Orthogonal decomposition]
Let $(G,\s)$ be a hierarchically hyperbolic group with clean containers which satisfies the
$\mathbf F_U$ stabilizers property, and let $h\in G$ be an infinite order element.  Let $\{U_1,\dots, U_{k+1}\}$ be a
maximal collection of pairwise orthogonal domains of $\s$ so that $\B(h)=\{U_1,\dots, U_k\}$ and $U_{k+1}$ is the container associated to $\B(h)$ in $S$.  Suppose
$h\in G$ fixes $\B(h)$ elementwise.  By Lemma~\ref{lem:choiceofFU}, there exists $t\in G$ and, for each $i=1,\dots, k$, a domain $U_i'\in\s'$ with $U_i=tU_i'$.   The label of the
vertex $\frak g_{tG_{U_i}}(h)$ is $th_i'$ for some $h_i'\in
G_{u_i'}\subseteq \FixedF_{U_i'}$.  Define
 \begin{equation}\label{eqn:decompterm}
 h_{U_i}:=th_i't^{-1}\in tG_{U_i'}t^{-1} = G_{U_i}.
 \end{equation}

 The group $(G,\s)$ satisfies the \emph{orthogonal decomposition property} if the following two properties hold for all axial elements $h\in G$.  First, there is a uniform lower bound on the translation length $\tau_{U_i}(h)$ for each $U_i\in \B(h)$ (this uniformity only depends  on the hierarchy constants and not  the choice of $h$).  Second, after possibly relabeling the domains of $\B(h)$, we have
 \[
 h=h_{U_1}h_{U_2}\ldots h_{U_k}=th_1'\ldots h_k't^{-1}.
 \]
 We say $h_{U_1}h_{U_2}\ldots h_{U_k}$ is a \emph{decomposition} of $h$. 
\end{defn}

This decomposition may depend on the order of the
factors.  In particular, it may be the case that $h_{U_i}$ does not
commute with $h_{U_j}$, because elements of $G_{U_i}$ and $G_{U_j}$
may not commute.  However, the final property we discuss 
will require that such elements do commute, and so the order of the factors will not be important for the groups we consider.

\subsubsection{The commutative property}

The final property ensures  that  $G_U$ and $G_V$ commute whenever $U\perp V$.
\begin{defn}[Commutative property]
A hierarchically hyperbolic group $(G,\s)$ with the $\mathbf F_U$ stabilizers property  satisfies the
\emph{commutative property} if $[G_U,G_V]=1$ whenever $U\perp V$.
\end{defn}

The following lemma is a consequence of the commutative property.

\begin{lem}\label{lem:specialaxial}
Let $G$ be a hierarchically hyperbolic group satisfying the $\mathbf F_U$ stabilizers, orthogonal decomposition, and commutative properties.  Let $h\in G$ be an axial element which fixes  $\B(h)=\{H_1,\dots, H_k\}$ elementwise, and let $C$ be the clean container associated to $\B(h)$.  Then there exists a uniform constant $K$ such that $(h^K)_C=1$, where $(h^K)_C$ is the factor corresponding to $C$ in the decomposition of $h^K$ with resepect to $\{H_1,\ldots,H_k,C\}$ and $1$ is the identity element of $G_C\leq G$.
\end{lem}

\begin{proof}
First, note that by Lemma \ref{lem:uniquecleancontainers}, $h$ fixes $\{H_1,\dots,H_k,C\}$ elementwise, and so the decomposition $h=h_{H_1}\ldots h_{H_k}h_C$ of $h$ with respect to this set is well-defined. 
Recall that $h_C\in G_C$ is an element of the hierarchically hyperbolic group $(G_C,\s_C)$.  Since $C\not\in \B(h)$, $h_C$ is not an axial element of $G_C$.  Therefore $h_C$ must be finite order by Lemma \ref{lem:finiteorder}.  By \cite[Theorem~G]{HHP:semihyp} there are finitely many conjugacy classes of finite order elements in a hierarchically hyperbolic group, and therefore there is a uniform constant $K$ such that $h_C^K$ is the identity element of $G_C$. 

By the commutative property, we have 
\[
h^K=(h_{H_1}\dots h_{H_k}h_C)^K=(h_{H_1})^K\dots (h_{H_k})^K(h_C)^K=(h_{H_1})^K\dots (h_{H_k})^K.
\]
From this decomposition, it is clear that $(h^K)_C=1$.
\end{proof}

\subsubsection{Examples}

We now give several examples of hierarchically hyperbolic groups
satisfying the three properties defined above.
Moreover, additional examples can be built 
using combination theorems, of which there are several in the
literature (see, e.g., \cite{BehrstockHagenSisto:HHS2,
BerlaiRobbio:combinationHHG, BerlyneRussell:HHSgraphproducts,
RobbioSpriano:2decomposableHHG}).

\begin{prop} \label{prop:specialHHGs}
Let $\Xi$ be the set of hierarchically hyperbolic groups with clean containers which satisfy the $\mathbf F_U$ stabilizers, orthogonal decomposition, and commutative properties.  Then the following groups are in $\Xi$.

	\begin{enumerate}
	\item Hyperbolic groups 
	\item Compact special groups 
	\item Groups hyperbolic relative to a collection of  groups in $\Xi$ 
	\item Direct products of groups in $\Xi$
	\end{enumerate}
\end{prop}

\begin{proof}  We consider each class of groups in turn.
\begin{enumerate}
\item The statement is immediate for hyperbolic groups $G$, as they all admit hierarchically hyperbolic structures with a single domain $S$, and the action on $\fontact S$ is acylindrical.  For this domain, $\mathbf F_S$ is a Cayley graph of the group and $G_S=G$.  As there is no orthogonality, the orthogonal decomposition and commutative properties vacuously hold.
\item For compact special groups, we use the standard structure
described in \cite{BehrstockHagenSisto:HHS2}.  This structure
satisfies the three properties by a completely analogous argument to
the one given for right-angled Artin groups in Example
\ref{ex:specialcommute}.

\item Let $G$ be a group which is hyperbolic relative to a collection
$\mathcal P$ of hierarchically hyperbolic groups with clean containers
satisfying the $\mathbf F_U$ stabilizers, orthogonal decomposition,
and commutative properties.  Then $G$ is a hierarchically
hyperbolic group by \cite[Theorem~9.1]{BehrstockHagenSisto:HHS2} and
has clean containers by
\cite[Proposition~7.4]{AbbottBehrstock:HHSlargest}.  For each $P\in
\mathcal P$, let $(P,\s_P)$ be an HHG structure for $P$, and for each
left coset $gP$, let $\s_{gP}$ be a copy of $\s_P$, with the
associated hyperbolic spaces and projections.  Let $\widehat G$ be the
hyperbolic space formed from $G$ by coning off each left coset of each
$P\in \mathcal P$.  Then the hierarchically hyperbolic group structure
on $G$ is given by $\s=\{\widehat G\}\sqcup_{gP\in G\mathcal
P}\s_{gP}$.  The domain $\widehat G$ is the unique
$\nest$--maximal domain, and if $U\in \s_{gP}$ and $V\in \s_{g'P'}$
where $gP\neq g'P'$, then $U\trans V$.  We refer the reader to
\cite[Section~9]{BehrstockHagenSisto:HHS2} for details of this
structure, but note one important feature of the structure $(G,\s)$: any
pair of orthogonal domains are contained in some $gP\in G\mathcal P$.

We first check that the $\mathbf F_U$ stabilizes property holds.  A
fundamental domain for the action of $G$ on $\s$ is given by
$\s'=\{\widehat G\}\sqcup _{P\in \mc P}\s_P$.  Let $U\in \s'$.  If 
$U=\widehat G$, then $\mathbf F_U=G$ and $G_U=G$, so the property holds
for this domain.  Now suppose $U\in \s_P$ for some $P\in\mathcal P$.
Then $\mathbf F_U\subseteq P$.  If $g\not \in P$, then $g\mathbf
F_U\subseteq gP\neq P$, and so $g\not \in G_U$.  Therefore $G_U$ is a
subgroup of $ P$ in this case.  Since $(P,\s_P)$ satisfies the
$\mathbf F_U$ stabilizers property, it follows that $(G,\s)$ does, as
well.

We now check the orthogonal decomposition property.  Since $G$ is
hyperbolic relative to $\mc P$, every infinite order element $h\in G$
is either loxodromic with respect to the action of $G$ on $\widehat
G$, in which case $\B(h)=\{\hat G\}$ or is conjugate into some $P\in
\mc P$, in which case we consider the conjugate $ghg^{-1}\in P$. 
In the first case, the action of $G$ on $\widehat
G$ is acylindrical, and so there is a uniform lower bound on the translation length of $h$, and we have the trivial orthogonal decomposition of $h$.   
In the second case, there is a uniform lower bound on the translation length of $ghg^{-1}$ in each domain in $\B(ghg^{-1})$ by the assumption that each $P$ satisfies the orthogonal decomposition property.  Translation length is invariant under conjugacy, and so we obtain a uniform lower bound on the translation length of $h$ in each domain in $\B(h)$.  There is also an orthogonal decomposition of
$ghg^{-1}$ coming from the assumption on $(P,\s_{P})$.   Since
$\B(h)=g^{-1}\B(ghg^{-1}$, conjugating each term in the
decomposition of $ghg^{-1}$ by $g^{-1}$ yields an orthogonal
decomposition for $h$.

Finally, the commutative property follows immediately from the
construction of the orthogonal decomposition in the previous paragraph
and the fact that $(P,\s_P)$ satisfies the commutative property for
each $P\in\mc P$.

\item Assume $G=G_1\times G_2$, and suppose $(G_1,\s_1),(G_2,\s_2)$
are hierarchically hyperbolic groups with clean containers which
satisfy the $\mathbf F_U$ stabilizers, orthogonal decomposition, and
commutative properties.  Then $G$ is a hierarchically hyperbolic
group by \cite[Proposition~8.27]{BehrstockHagenSisto:HHS2} and has
clean containers by
\cite[Proposition~7.3]{AbbottBehrstock:HHSlargest}.  The hierarchy
structure on $G$ is given by $\s=\{S,
U_1,U_2\}\sqcup\s_1\sqcup\s_2\sqcup \{V_U\mid U\in \s_1\cup\s_2\}$,
where $S$ is the unique $\nest$--maximal element, $U_i$ is a domain
into which all domains in $\s_i$ nest, and for each $U\in \s_i$, the
domain $V_U$ is a domain into which all domains in $\s_{j}$ with
$j\neq i$ and all domains in $\s_i$ orthogonal to $U$ nest.  The only
important relation between domains for this proof is orthogonality.
In addition to any orthogonality among domains in $\s_1$ or $\s_2$, we
have that all domains in $\s_1$ are orthogonal to all domains in
$\s_2$, $U_1\perp U_2$, and $V_U\perp U$ for each $U\in \s_1\cup\s_2$.
By construction, $(G,\s)$ has clean containers. 
See \cite[Section~8]{BehrstockHagenSisto:HHS2} for further details on
this structure.

When we refer to subsets of $G_i$ or the $(G_i,\s_i)$ structure, we 
append a superscript $i$ to the notation.  For example, if $U\in
\s_i$, then $\mathbf F_U^i$ is the corresponding subset of $G_i$.

We first check the $\mathbf F_U$ stabilizers property.  If $U=S$,
there is nothing to check, so suppose first that $U\in \s_1$.  Let
$G_U^1$ denote the subgroup from the structure $(G_1,\s_1)$ which
stabilizers $\mathbf F_U^1\times \{e\}$ for each $e\in \mathbf E_U^1$.
In the structure $(G,\s)$, there are additional domains orthogonal to
$U$; in particular every domain in $\s_2$ is orthogonal to $U$.  We
have $\mathbf F_U=\mathbf F_U^1$, but now $\mathbf E_U=\mathbf E_U^1
\times G_2$.  Therefore, we have $(g_1,g_2)\in G_U$ if and only if
$g_1\in G_U^1$ and $g_2=1$.  Thus $G_U\simeq G_U^1\times \{1\}$.
Since $(G_1,\s_1)$ satisfies the $\mathbf F_U$ stabilizers property,
$G_U^1$ is coarsely equal to $\mathbf F_U^1$.  The above discussion
then implies that $G_U$ is coarsely equal to $\mathbf F_U$.
Similarly, if $U\in \s_2$, then $G_U\simeq \{1\}\times G_U^2$, and we
again have that $G_U$ is coarsely equal to $\mathbf F_U$.

Suppose next that $U=U_1$.  Then $\mathbf F_U=G_1$, and $\mathbf
E_U=G_2$.  Since $G$ is the direct product of $G_1$ and $G_2$, we have
that $G_{U_1}=G_1$, and so $G_{U_1}$ is coarsely equal to $\mathbf
F_U$.  The analogous argument holds if $U=U_2$.

Finally, fix $U\in\s_1$, and consider the domain $V_U$.  Let $C_U$ be
the container associated to $U$ in the $\nest$--maximal domain of
$\s_1$.  Then $\mathbf F_{V_U}=\mathbf E_U=\mathbf E_U^1\times
G_2=\mathbf F_{C_U}^1\times G_2$, and $\mathbf E_{V_U}=\mathbf E^1_{C_U}$.  It
follows that $G_{V_U}\simeq G^1_C\times G_2$.  Since $\mathbf F^1_{C_U}$
is coarsely equal to $G_{C_U}^1$, we also have that $\mathbf F_{V_U}$ is
coarsely equal to $G_{V_U}$, as desired.  An analogous argument holds
if we fix $U\in \s_2$.  Therefore, $(G,\s)$ satisfies the $\mathbf
F_U$ stabilizers property.

The orthogonal decomposition and commutative properties both 
follow immediately because they hold in each $(G_i,\s_i)$ and 
 $G_1$ and $G_2$ commute.\qedhere
\end{enumerate}
\end{proof}

The $\mathbf F_U$ stabilizers, orthogonal decomposition, and
commutative properties all involve orthogonality and properties of
product regions.  Hence, intuitively, if a combination theorem does
not add any additional orthogonality relations (or only in a trivial
way, such as by adding domains whose associated hyperbolic space is
bounded diameter), then such a combination of groups in
$\Xi$ should, in general, yield a group in $\Xi$.  For example, we expect
that trees of groups in $\Xi$ satisfying the hypotheses of the combination theorem in \cite[Theorem~8.6]{BehrstockHagenSisto:HHS2} are also in $\Xi$.  In particular,
combined with Proposition \ref{prop:specialHHGs} (3) \& (4), this
would show that for hierarchically hyperbolic groups $\pi_1(M)$ 
where $M$ is the fundamental group of compact 3--manifolds with no Nil
or Sol in its prime decomposition, then $\pi_1(M)$ is in $ \Xi$.

\subsection{A non-example: the mapping class group} We  briefly explain why the standard hierarchy structure on the
mapping class group fails to satisfy the $\mathbf F_U$ stabilizers
property.  Notwithstanding this fact, we believe that a modification of
the properties of this section can be used to make the present approach
work for the mapping class group, as well. We do not carry this out,
though, because the approaches we see for doing so are all 
technical, and the present results are already
known for mapping class groups.  We record this fact  for those using
these properties in the future with an eye towards other applications.

The standard hierarchically hyperbolic group structure $\s$ on the
mapping class group of a surface $S$ is described in
\cite[Theorem~11.1]{BehrstockHagenSisto:HHS2}.  The domains $U\in\s$
correspond to homotopy classes of essential, not necessarily
connected, \emph{open} subsurfaces $U\subseteq S$.  Two domains are
orthogonal if the corresponding subsurfaces are disjoint.  In
particular, the annuli about 
the boundary curves of a subsurface do not intersect the
subsurface; thus an annulus around a boundary curve is a domain orthogonal
to the subsurface.  A finite fundamental domain $\s'$ for the action
of $\MCG(S)$ on $\s$ is provided by taking a collection of
subsurfaces, one for each homeomorphism type of subsurface.  For each
$U\in \s'$, $\mathbf F_U$ is coarsely equal to the mapping class group
of the subsurface associated to $U$ and $\mathbf E_U$ is coarsely
equal to the mapping class group of the complementary closed subsurface
$S-U$.

One subtlety in the hierarchically hyperbolic structure on mapping
class groups is that while elements of $\MCG(S)$ supported on disjoint
subsurfaces commute, elements supported on disjoint \emph{closed} 
subsurfaces are distinct, while two elements supported on 
disjoint \emph{open} surfaces may coincide. A simple example of this 
is found by taking a product of elements in a once punctured torus 
which generate the Dehn twists along the boundary.   Taking the 
genus two surface obtained by doubling along the boundary curve, we 
see that we can generate that same Dehn twist by a product of 
elements on either of the open once-punctured tori separated by that 
curve.

Associated to a closed subsurface $V$, which includes its boundary
components, is an element of $\s$ consisting of the disjoint union of
the interior of $V$, which we will denote $\mathring V$, with annuli
around the elements $\alpha_1,\dots, \alpha_k$ of $\partial V$.  The Dehn
twist about a boundary curve in $\partial V$ can be represented as 
a product of 
mapping class elements supported on the interior of $V$, even those
these are orthogonal domains.  Accordingly the stabilizer of 
$\mathring V$, in the action of $G$
on $\s$, is (possibly up to  finite index  if $V$ is homeomorphic 
to $S-V$) a central extension of 
$\MCG(\mathring V)\times \MCG(S-V)$ by 
$\mathbb Z^k$, where $\mathbb Z^k$ is generated by Dehn twists along
the boundary curves $\alpha_i$, see, e.g.,
\cite{BirmanLubotzkyMcCarthy}.  The domains $\mathring V,S-V,$ and the
annuli around each $\alpha_i$ form a maximal collection of pairwise
orthogonal domains. If this was a semidirect product instead of a 
central extension, this would yield the $\mathbf F_U$ stabilizers
and orthogonal decomposition 
properties.  However, the fact that $\MCG(\mathring V)$ doesn't act 
cocompactly on $\mathbf F_U$ means that the $\mathbf F_U$ stabilizers 
property doesn't hold in this structure.

We note, though, that any open subsurface $\mathring U$ is contained in 
a larger subsurface $U$ obtained by taking the union of $\mathring U$ 
and all the annuli which bound $\mathring U$.  For this subsurface $U$, 
the subgroup $G_U$ of the mapping class group of $S$ which
stabilizes $\mathbf F_U$ and fixes $\mathbf E_U$ pointwise \emph{can} be
identified with $\MCG(U)$. This is a weaker version of 
the $\mathbf F_U$ stabilizers property. We expect that this weaker 
version might be useful in future work.\footnote{We note that a 
related property to this is studied in forthcoming work of Montse 
Casals-Ruiz, Mark Hagen, and Ilya Kazachkov.}

\subsection{Conjugators in HHGs}
We are now ready to prove Theorem \ref{thmi:SHHGconjugators}, which we restate for the convenience of the reader.

\Restate{Theorem}{thmi:SHHGconjugators}{}{Let $(G,\s)$ be a
hierarchically hyperbolic group satisfying the $\mathbf F_U$
stabilizers, orthogonal decomposition, and commutative properties.
There exist constants $K,C$ and $N$ such that if $a,b\in G$ are
infinite order elements which are conjugate in $G$, then there exists $g\in G$ with $ga^N=b^Ng$ and
$$
|g|\leq K(|a|+|b|)+C.
$$ }

\begin{proof} Fix a hierarchically hyperbolic group $(G,\s)$ and a
finite fundamental domain $\s'$ for the action of $G$ on $\s$ as at
the beginning of this section. Assume that  $(G,\s)$ satisfies the 
$\mathbf F_U$
stabilizers, orthogonal decomposition, and
commutative properties.
 For each $U\in \s$, we fix $\FixedF_U=\mathbf F_U\times\{e\}$ as described in Lemma \ref{lem:choiceofFU}.

We fix the same constants as in the beginning of the proof of Theorem~\ref{mainthm}, and let $\sigma$ be the Morse constant for $(\lambda,\lambda)$--quasi-geodesics in a $\delta$--hyperbolic space.  Fix the function $k'\colon [0,\infty)\to[0,\infty)$ so that $G_U$ is $k'$--hierarchically quasiconvex whenever $U\in \s'$, and  let $A$ be the constant from Lemma \ref{lem:gatesequivar} applied to $k'$--hierarchically quasiconvex subspaces.   We further increase $R$ so that $R>\max\{3E+A,E+E\nu+A+\nu,s_0\}$ and $K$ so that 
\begin{equation}\label{eqn:newK}
K=\max\{2\delta, 3R, K_0, \left\lceil\frac{4L+4\delta+E}{T}\right\rceil+2, 2L+2\delta,\frac{6E + A +\sigma + 1}{T},3E+2\sigma\}.
\end{equation}

Let $a,b\in G$ be two infinite order  elements, and suppose there
exists $g\in G$ such that $ga=bg$.   Then $g\B(a)=\B(b)$.  Let $C$ be the container associated to $\B(b)$ in $S$, so that $\B(b)\cup\{C\}=\{B_1,\dots, B_k,C\}$ is a maximal collection of pairwise orthogonal domains.  Since $g$ conjugates $a^i$ to $b^i$ for any $i$, we first replace $a$ and $b$ by sufficiently high powers so that the following conditions are satisfied:
\begin{enumerate}[(a)] 
\item $\B(a)$ and $\B(b)$ are fixed pointwise by $a$ and $b$, respectively; 
\item $b$ has the decomposition $b= b_1\cdots  b_k$ with respect to $\B(b)=\{B_1,\dots, B_k, C\}$, where $ b_i=b_{B_i}$ is as in \eqref{eqn:decompterm};  and
\item  $\tau_V(b)\geq L_S\delta$ for every $V\in\B(b)$.   
\end{enumerate}
 Such powers exist and can be chosen uniformly (that is, depending
 only on the hierarchy constants, and not the choice of elements $a$
 and $b$) by the discussion after Definition \ref{def:axial} in the
 first case, the orthogonal decomposition property and Lemma
 \ref{lem:specialaxial} in the second case, and the assumed bound on 
 translation length in the orthogonal decomposition property
 in the third case.  Lemma \ref{lem:choiceofFU}
 applied to $\B(b)$ provides an element $t\in G$ so that $b_i=tb_i'$,
 where $b_i'\in G_{U_i'}\subseteq \FixedF_{U_i'}$, where $U_i'\in \s'$
 and $U_i=tU_i'$ for all $i=1,\dots, k$.

For each $Z\in\B(b)$, let $\gamma_b^Z$ be a $(2,\ell)$--quasi-geodesic
axis of $b$ in $\fontact Z$.  Then $\gamma_a^{g^{-1}Z}\colon= 
g^{-1}\gamma_b^{Z}$ is an $(2,\ell)$--quasi-geodesic axis of $a$ in
$\fontact (g^{-1}Z)$.
We now fix a quadrilateral of $(\lambda,\lambda)$--hierarchy paths
$\mu(1,g)$, $\mu(1,b^K)$, $b^K\mu(1,g)=\mu(b^K,b^Kg)$, and
$g\mu(1,b^K)=\mu(g,gb^K) =\mu(g,a^Kg)$ in $G$.

\subsubsection*{Step 1: Changing the conjugator.} Our first step is
to replace $g$ with a (possibly) different conjugator whose length we are able to bound in $G$.  We will do this by first premultiplying $g$ by a power of $b_i\in G_{B_i}$ for each $B_i\in \B(b)$.  By the commutative property, any power of $b_i$ commutes with $b$, and so this new element will still conjugate $a$ to $b$.  This is analogous to how we changed the conjugator in the proof of Theorem~\ref{mainthm}, when $\B(a)=\B(b)=\{S\}$.  In that situation, the orthogonal decomposition of $b$ was simply $b=b_S$, and we premultiplied the conjugator by a power of $b$.   In the current situation we need to be a bit more careful because not only may $b$ have more than one term in its orthogonal decomposition, but now $\B(a)$ and $\B(b)$ may be \emph{different} collections of domains.  Because of this, we will need to estimate distances in multiple hyperbolic spaces.  Finally, we will alter $g$ in the clean container $C$ associated to $\B(b)$. 

Fix $Z\in\B(b)$ and let $b_Z=\frak g_{G_Z}(b)$.  Since
$K\geq K_0$, we may apply Lemma \ref{axiallem} to each of the axes
$\gamma_b^Z$ in $\fontact Z$ and $\gamma_a^{g^{-1}Z}$ in $\fontact
(g^{-1}Z)$, and the points $1,a^K\in G$ and $1,b^K\in G$, respectively.
This yields a point $z'\in \gamma^{g^{-1}Z}_a$ and a point $w'\in
\gamma^Z_b$ such that $z'\in\mathcal N_{L}(\pi_{g^{-1}Z}(\mu(1,a^K)))$
and $w'\in \mathcal N_{L}(\pi_{Z}(\mu(1,b^K)))$, where these
neighborhoods are taken in $\fontact g^{-1}Z$ and $\fontact Z$,
respectively.  Moreover, if $x$ is a nearest point on
$\gamma^{g^{-1}Z}_a$ to $\pi_{g^{-1}Z}(1)$ in $\fontact(g^{-1}Z)$ and
$y$ is a nearest point on $\gamma^Z_b$ to $\pi_Z(1)$ in $\fontact
Z$, then $\dist_{g^{-1}Z}(x,z')\leq L$ and $\dist_Z(y,w')\leq L$.  Let
$z\in \pi_{g^{-1}Z}(\mu(1,a^K))$ and $w\in\pi_Z(\mu(1,b^K))$ be 
nearest points to $z'$ and $w'$, respectively, so that
$\dist_{g^{-1}Z}(z,z')\leq L$ and $\dist_Z(w,w')\leq L$.  See Figure
\ref{fig:ThmDCase1}.

Since the isometry $g$ maps $\fontact (g^{-1}Z)$ to $\fontact Z$, we have $gz'\in g\gamma_a^{g^{-1}Z}=\gamma_b^Z$ and $gz\in g\pi_{g^{-1}Z}(\mu(1,a^K))=\pi_Z(\mu(g,ga^K))=\pi_Z(\mu(g,b^Kg))$.  Moreover,  $\dist_Z(gz,gz')=\dist_{g^{-1}Z}(z,z')\leq L$ and $\dist_Z(gz',gx)=\dist_{g^{-1}Z}(z',x)\leq L$.

\begin{figure}[h]
  \centering 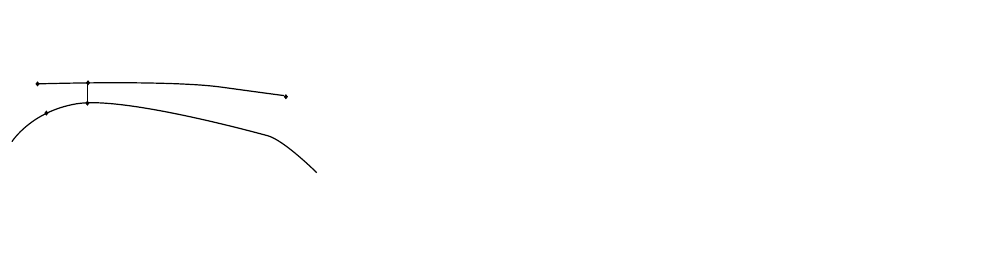
	\caption{The geometry of the axes of $a$ and $b$ in $\fontact g^{-1}Z$ and $\fontact Z$, respectively.} 
	\label{fig:ThmDCase1}
\end{figure}

By possibly premultiplying $g$ by a power of $ b_Z$, we may assume that
$\dist_Z(gz',w')\leq \tau_Z(b)$.  Moreover,  this new element also conjugates $a$ to $b$, because $ b_Z$ commutes with $b$ by the  commutative property.  

We perform the above procedure for each $Z\in\B(b)$ and possibly premultiply $g$ by a (possibly different) power $m_Z$ of each $ b_Z$.   

We now alter $g$ in the clean container $C$ associated to
$\B(b)$.  Let $t\in G$ be as in Lemma~\ref{lem:choiceofFU} applied to $\{B_1,\dots, B_k,C\}$, so that $C=tC'$ for some $C'\in \s'$.  The label of the vertex  $\frak g_{tG_C}(g)$ is $tg_{C'}$, where $g_{C'}\in G_{C'}\subseteq \FixedF_{C'}$.  Let $g_C:= tg_{C'}t^{-1}\in G_C$.

We claim
that $g_C^{-1}g$ conjugates $a$ to $b$ and $\dist_V(1,g_C^{-1}g)\leq
A$ for each $V\nest C$.  The commutative property and
condition (b) ensure that $ g_C^{-1}$ commutes with $b$, hence
$g_C^{-1}g$ conjugates $a$ to $b$.

We have 
\[
\frak g_{\FixedF_{C}}(g_C^{-1}g)\asymp_A g_C^{-1}\mathfrak g_{g_C\FixedF_C}(g)=tg_{C'}^{-1}t^{-1}\frak g_{\FixedF_C}(g),\]
where the first estimate follows from Lemma \ref{lem:gatesequivar} and the second from the definition of $g_C$ and the fact that $g_C\FixedF_C=\FixedF_C$. 

By the $\mathbf F_U$ stabilizers property, $\dist_G(\frak g_{\FixedF_C}(g), \frak g_{G_C}(g))\leq \nu$.  We also have 
\[ tg_{C'}^{-1}t^{-1}\frak g_{G_C}(g)=tg_{C'}^{-1}t^{-1}(tg_{C'})=t.\]
Thus 
\begin{equation}\label{eqn:distttogate}
\dist_G(t,\frak g_{\FixedF_C}(g_C^{-1}g))\leq \dist_G(t, tg_{C'}^{-1}t^{-1}\frak g_{\FixedF_C}(g)) + \dist_G(tg_{C'}^{-1}t^{-1}\frak g_{\FixedF_C}(g),\frak g_{\FixedF_C}(g_C^{-1}g)) \leq \nu+A.
\end{equation}
 By \cite[Remark~1.16]{BehrstockHagenSisto:asdim} and Remark \ref{rem:gates}, we have $\pi_V(\frak g_{\FixedF_C}(g_C^{-1}g))=\pi_V(g_C^{-1}g)$. Since the projection maps $\pi$ are Lipschitz, it thus follows from \eqref{eqn:distttogate} that $\dist_V(t, g_C^{-1}g)\leq A+\nu$ for all $V\nest C$.

By Lemma \ref{lem:choiceofFU}, we have $\dist_G(t,\frak g_{\mathbf P_{\mc U}}(1))\leq E\nu$, where $\mc U=\{B_1,\dots, B_k,C\}$.  
The only domains which are $E$--relevant for $1,\frak g_{\mathbf P_{\mc U}}(1)$ are those which are transverse to some element of $\mc U$ or into which some element of $\mc U$ properly nests by Lemma \ref{lem:disttoprod}.  In particular,  $\dist_V(1,\frak g_{\mathbf P_{\mc U}}(1))\leq E$ for all $V\nest C$.  By the triangle inequality and the fact that the maps $\pi_U$ are Lipschitz,   we have for all $V\nest C$
\[
\dist_V(1,g_C^{-1}g) \leq \dist_V(1,\frak g_{\mathbf P_{\mc U}}(1)) + \dist_V(\frak g_{\mathbf P_{\mc U}}(1), t) + \dist_V(t,g_C^{-1}g) 
 \leq E + E\nu + A + \nu  <R.
\]

This yields a new element $\left(\prod_{Z\in
\B(b)}b_Z^{m_Z}\right)g_C^{-1}g$, which also conjugates $a$ to $b$. 
 We have shown that this new conjugator, which by an abuse of notation we still call $g$, satisfies the following properties:  
\begin{equation}\label{eqn:tauUb}
\dist_Z(y,gx)\leq \dist_Z(y,w')+\dist_Z(w',gz')+\dist_Z(gz',gz)\leq \tau_Z(b)+2L
\end{equation}
for each $Z\in\B(b)$, and 
\begin{equation}\label{eqn:Vbound}
\dist_V(1,g) <R
\end{equation}
whenever $V\nest C$.

\subsubsection*{Step 2: Bounding the length of $g$}
Our goal is to bound the length of $g$ in $G$. As in the proof of Theorem \ref{mainthm}, we will show that for each $U\in\s$, we have 
\begin{equation}\label{eqn:UboundThmD}
\dist_U(1,g)\leq
2K\dist_U(1,b)+\dist_U(g,b^Kg)+K,
\end{equation} where $K$ is as in \eqref{eqn:newK}.  After establishing this bound for
each $U\in \s$, we then apply the distance formula with threshold $R$ and the fact that
$\dist_G(g,b^{K}g)=\dist_G(g,g a^{K})=\dist_G(1,a^{K})$, which, as in the proof of Theorem \ref{mainthm},
establishes that
\begin{equation} \label{eq:step2case2}
\dist_G(1,g)\preceq_{C,D} 2K\dist_G(1,b)+K\dist_G(1,a)+K,
\end{equation}
where $C,D$ are the constants given by the distance formula (Theorem
\ref{thm:distance_formula}).  (Note that by assumption, $R$ is 
sufficiently large to serve as a threshold in the distance formula.)  
This will provide the desired bound in $G$.

\medskip

Fix $U\in\s$. If $U\not\in \relevant(1,g;R)$, then we have $\dist_U(1,g)\leq R\leq K$, and \eqref{eqn:UboundThmD} holds.  Thus we assume for the rest of the proof that $U\in \relevant(1,g;R)$. There are five cases to consider:  there is some $Z\in\B(b)$ such that $U=Z$; there is some $Z\in\B(b)$ such that $U\propnest Z$; there is some $Z\in\B(b)$ such that $U\sqsupsetneq Z$; and there is some $Z\in\B(b)$ such that $U\trans Z$, and $U\perp Z$ for all $Z\in\B(b)$.  \\

\noindent {\bf Cases 1 and 2:} There is some $Z\in\B(b)$ such that 
$U=Z$ or $U\propnest Z$. 

These two cases follow almost exactly as in proof of 
Theorem \ref{mainthm}, the distinction being that $Z$ plays the role of $S$ and we measure distances in both $\fontact (g^{-1}Z)$ and $\fontact Z$.  In Case 2, 
one must also use  
\eqref{eqn:tauUb} in place of \eqref{eqn:taub}. \\

\noindent {\bf Case 3:} There is some $Z\in\B(b)$ such that $U\sqsupsetneq Z$.

By our choice of $K$,  we have $KT\geq E$, and thus $\dist_Z(1,b^K)\geq E$. Applying the bounded geodesic image axiom (Definition \ref{defn:HHS} (\ref{item:dfs:bounded_geodesic_image})) to $\pi_U(\mu(1,b^K))$ in $\fontact U$, we obtain $\rho^Z_U\subseteq\mathcal N_{E+\sigma}(\pi_U(\mu(1,b^K)))$ in $\fontact U$, and hence 
\begin{equation}\label{eq:1rhoZU}
\dist_U(1,\rho^Z_U)\leq \dist_U(1,b^K)+E+\sigma.
\end{equation}  
Additionally, $g^{-1}Z\in \B(a)$ and $g^{-1}U\sqsupsetneq g^{-1}Z$.  The choice of $K$ ensures that  $d_{g^{-1}Z}(1,a^K)\geq E$, so applying the bounded geodesic image axiom to $\pi_U(\mu(1,a^K))$ in $\fontact (g^{-1}U)$ yields 
\[
\rho^{g^{-1}Z}_{g^{-1}U}\subseteq \mathcal N_{E+\sigma}(\pi_U(\mu(1,a^K)))
\]
 in $\fontact (g^{-1}U)$.    Applying the isometry $g$ we obtain 
 \[
 g\rho^{g^{-1}Z}_{g^{-1}U}\subseteq \mathcal N_{E+\sigma}(g\pi_{g^{-1}U}(\mu(1,a^K)))=\mathcal N_{E+\sigma}(\pi_U(\mu(g,b^Kg)))
 \]
  in $\fontact U$. See Figure \ref{fig:case3}.
 \begin{figure}
 \def\svgwidth{2.75in}
  \centering  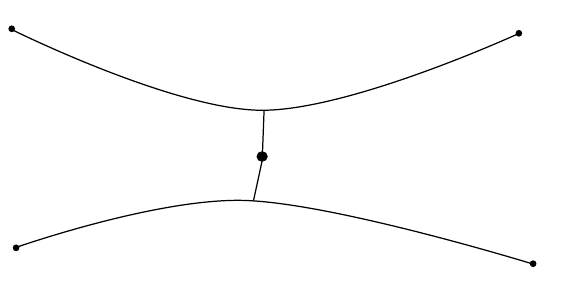
	\caption{Case 3.} 
	\label{fig:case3}
\end{figure} 
  Moreover,  projection maps in a hierarchically hyperbolic group are $G$--equivariant, and so $g\rho^{g^{-1}Z}_{g^{-1}U}=\rho^Z_U$. Thus 
  \begin{equation}\label{eq:grhoZU}
  \dist_U(g,\rho^Z_U)=\dist_U(g, g\rho^{g^{-1}Z}_{g^{-1}U})\leq \dist_U(g,b^Kg)+E+\sigma.
  \end{equation}
      Therefore, by the triangle inequality, \eqref{eq:1rhoZU}, and \eqref{eq:grhoZU}, we have
    \begin{align*}
    \dist_U(1,g)&\leq \dist_U(1,\rho^Z_U) + \diam_{\mc CU}(\rho^Z_U) + \dist(\rho^Z_U,g) \\
    & \leq \dist_U(1,b^K)+3E+2\sigma+\dist_U(g,b^Kg) \\
    &\leq K\dist_U(1,b) +\dist_U(g,b^Kg)+K,
    \end{align*}
where the final inequality follows  because $K\geq 3E+2\sigma$.\\

\noindent {\bf Case 4:} There is some $Z\in\B(b)$ such that $U\trans Z$.

Consider the product region $\mathbf P_Z$, and let 
$\xi=\gate_{\mathbf{P}_Z}(g)$ and $\nu=\gate_{\mathbf P_{Z}}(g)$.
 See Figure \ref{abconjugate}.  
\begin{figure} [h]
\def\svgwidth{4in}  
  \centering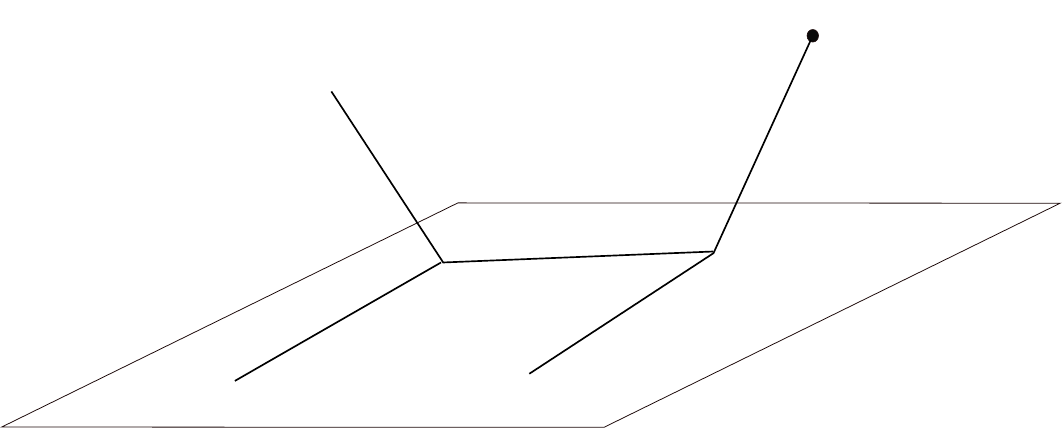 
	\caption{Conjugate elements $a$ and $b$, with conjugator $g$, in $G$.  Solid segments are hierarchy paths, while dotted segments are geodesics.} \label{abconjugate}
\end{figure}

Since we are assuming that $U$ is relevant for $1,g$ and $U\trans Z$, it follows from Lemma \ref{lem:disttoprod} that
$U\in\relevant(1,\xi;R)\cup \relevant(g,\nu;R)$.  As $b$ is loxodromic with respect to the action on
$\fontact Z$ for all $Z\in\B(b)$, we have $\dist_Z(\nu,b^K\nu)\geq
K\tau_Z(b)\geq KT\geq R$.  Thus 
$Z\in\relevant(\nu,b^K\nu; R)$.  Similarly, $Z\in \relevant (\xi,b^K\xi; R)$.  Note that this implies 
$Z\in \relevant(1,b^K;R)\cap\relevant(1,b^K\xi;R)\cap \relevant(g,b^Kg;R)\cap\relevant(g,b^K\nu;R)$, as well.

\begin{claim}\label{claim:Case4}
If $U$ is $R$--relevant for $g,g\nu$, then $U$ is not $R$--relevant for $b^K\nu,b^Kg$.  If $U$ is $R$--relevant for $1,\xi$, then $U$ is not $R$--relevant for $b^K\xi,b^K$.  
\end{claim}

\begin{proof}
We will prove the first statement.  The proof of the second statement is completely analogous.

Since $U$ is relevant for $g,\nu$, the domain $b^KU\in\s$ is relevant
for $b^K\nu,b^Kg$.  Moreover, since $b$ fixes $\B(b)$ pointwise,
we have $b^KZ=Z$.  As $U\trans Z$, we must also have 
$b^KU\trans Z$. From this, we apply the $G$--equivariance of the 
projections maps in a hierarchically 
hyperbolic group to conclude that 
$b^K\rho^U_Z = \rho^{b^KU}_Z$  in $\fontact Z$. 
Thus we have 
\[
\dist_Z(\rho^U_Z,\rho^{b^KU}_Z)=
\dist_Z(\rho^U_Z,b^K\rho^U_Z)\succeq_E K\tau_Z(b)\geq KT.
\]

Since $Z\in \B(b)$, Lemma \ref{lem:disttoprod} implies that $Z$ is not
$s$--relevant for $g,\nu$ for any $s\geq s_0$.  In particular, since $E\geq s_0$, the distance 
between $\pi_Z(g)$ and $\pi_Z(\nu)$ in $\fontact Z$ is bounded by 
$E$.  On the other hand, since $U$ is $R$--relevant for $g$ and 
$\nu$, it follows from \cite[Proposition~5.17]{BehrstockHagenSisto:HHS2} that any hierarchy path  $\mu(g,\nu)$ in $G$ has  a subpath which is contained in the $E$--neighborhood of $\mathbf P_U$.    Since the projection maps $\pi$ are Lipschitz, we have
\begin{equation}\label{eqn:Case4}
\dist_Z(\pi_Z(\mathbf P_U),\pi_Z(\mu(g,\nu)))\leq E.
\end{equation}

Recall that  $\rho^U_Z\asymp_E\pi_Z(\mathbf P_U)$ (see comments after Definition \ref{def:productregion}).  Thus 
\eqref{eqn:Case4} implies:  
\[
\dist_Z(\rho^U_Z,\pi_Z(\mu(g,\nu)))\leq 2E.
\]
Since $\pi_Z(\mu(g,\nu))$ is an unparametrized $(\lambda,\lambda)$--quasigeodesic, it is contained in the $\sigma$--neighborhood of a geodesic in $\fontact Z$ from $\pi_Z(g)$ to $\pi_Z(\nu)$.  By the above discussion, such a geodesic necessarily has length at most $E$.  Therefore,
\[
\dist_Z(\rho^U_Z,g)\leq 3E+\sigma.
\]

By the triangle inequality, we have 
\begin{align*}
\dist_Z(\rho^U_Z,b^Kg)&\geq \dist_Z(g,b^Kg)-\dist_Z(\rho^U_Z,g) \\
& \geq K\tau_Z(b)-\dist_Z(\rho^U_Z,g) \\
& \geq KT-(3E+\sigma)\\
& > 3E + A,
\end{align*}
where the final inequality follows from our choice of 
$K\geq\frac{6E + A +\sigma+1}{T}$. See Figure \ref{fig:Case4}.
\begin{figure}
  \centering 
  \def\svgwidth{3in}
  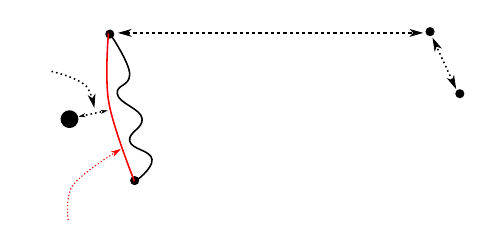
	\caption{The arrangement of points in $\mc CZ$ in the proof of Claim \ref{claim:Case4} in Case 4.} 
	\label{fig:Case4}
\end{figure}
Therefore $\dist_Z(\rho^U_Z,b^Kg)$ is large enough to  apply the
consistency inequalities (Definition \ref{defn:HHS} (\ref{item:dfs_transversal})), yielding 
\begin{equation}\label{eq:b^kgU}
\dist_U(\rho^Z_U,b^Kg)\leq E.\end{equation}

  The same argument bounding the distance in $\fontact Z$ between $g$ 
and $\nu$ applies to show that $Z$ is not $E$--relevant for $b^kg,\frak g_{\mathbf P_Z}(b^Kg)$.  By Lemma \ref{lem:gatesequivar}, we have $b^K\nu\asymp_A \frak g_{\mathbf P_Z}(b^Kg)$.  Therefore $Z$ is not $(E+A)$--relevant for $\nu,b^k\nu$, and so $\pi_Z(b^Kg)\asymp_{E+A} \pi_Z(b^K\nu)$ in $\fontact Z$.  It follows that 
\[\dist_Z(\rho^U_Z,b^K\nu)\asymp_{2E+A} \dist_Z(\rho^U_Z,b^Kg) > 3E + A,\] from which we conclude $\dist_Z(\rho^U_Z,b^K\nu) > E$. Thus we may again apply the consistency inequalities, yielding 
\begin{equation}\label{eqn:rhoZUbKgnu}
\dist_U(\rho^Z_U,b^K\nu)\leq E.
\end{equation}
  Combining this with
 \eqref{eq:b^kgU} and applying the triangle inequality yields
\[
\dist_U(b^Kg,b^K\nu)\leq \dist_U(b^Kg,\rho^Z_U) + \diam_{\mc CU}(\rho^Z_U) + \dist_U(\rho^Z_U,b^K\nu) \leq E+E+E =3E.
\]
Since $R>3E$, we have $U\not\in\relevant(b^K,b^K\nu;R)$.  This completes the proof of the claim.
\end{proof}

Suppose first that $U\in\relevant(1,\xi;R)\cap \relevant(g,\nu;R)$.  Then by the claim, we have that $U\not\in\relevant(b^K,b^K\xi;R)\cup \relevant(b^Kg,b^K\nu;R) $.  By Lemma \ref{lem:disttoprod} and the fact that $U\trans Z$, this is equivalent to 
$U\not\in\relevant(b^K,b^Kg;R)$.  Thus 
$\dist_U(b^K,b^{K}g)<R$, and so  we have:

\[
\dist_U(1,g)\leq \dist_U(1,b^{K})+\dist_U(g,b^{K}g)+R
\leq K\dist_U(1,b)+\dist_U(g,b^{K}g)+K.
\]

Now suppose that $U\not\in\relevant(1,\xi;R)\cap \relevant(g,\nu;R)$.  Since $U\in\relevant(1,\xi;R)\cup \relevant(g,\nu;R)$, we must have either $U\in\relevant(1,\xi;R)$ or $U\in\relevant(g,\nu;R)$. Suppose without loss of generality that $U\in \relevant(1,\xi;R)$ but $U\not\in\relevant(g,\nu;R)$.  It follows from the claim that $U\not\in\relevant(b^K,b^K\xi;R)$.  Moreover, by Lemma \ref{lem:disttoprod}, we have $\dist_U(b^K\xi,\nu)\leq R$.   Therefore:

\begin{align*}
\dist_U(1,g) & \leq \dist_U(1,b^K)+\dist_U(b^K,b^K\xi)+\dist_U(b^K\xi,\nu)+\dist_U(\nu,g) \\
 & \leq K\dist_U(1,b)+R+R+R \\ 
  & \leq K\dist_U(1,b)+K,
\end{align*}
where the final inequality follows because $K\geq 3R$. 
Thus \eqref{eqn:UboundThmD} holds regardless of whether $U\in\relevant(1,\xi;R)\cap \relevant(g,\nu;R)$.\\

\noindent {\bf Case 5: $U\perp Z$ for all $Z\in\B(b)$.}
Note that   $U\perp Z$ for all $Z\in \B(b)$ if and only if $U\nest C$, where $C$ is the container associated to $\B(b)=\{B_1,\dots, B_k\}$.  Thus the bound $\dist_U(1,g)\leq R$ follows immediately from \eqref{eqn:Vbound}.

This completes the proof of the theorem.
\end{proof}

\begin{rem}
Theorem \ref{thmi:SHHGconjugators} establishes the linear conjugator
property for suitable \emph{powers} of pairs of conjugate
infinite order elements.  In particular, the conjugator whose length
we bound in these theorems may not conjugate $a$ to $b$.  There are
two additional steps necessary to extend the ideas in these proofs to
show the linear conjugator property holds for \emph{all} pairs of
conjugate infinite order elements.  First, one would have to deal with
the fact that an element may permute the elements in its big set, an
issue we avoid by passing to a power to assume that the big set is
fixed elementwise.  This is likely not a serious problem.  Second, one
would need to understand the conjugator length function for finite
order elements.  Recall that in the decomposition of $b$ in the proof
of Theorem \ref{thmi:SHHGconjugators}, the factor $b_C$ corresponding
to the container associated to the big set of $b$ was a finite order
element of the corresponding sub-hierarchically hyperbolic group
$(G_C,\s_C)$.  We passed to a power so that we could assume this
factor was trivial.  If we don't pass to a power, we need a different
way to modify the conjugator in that sub-hierarchically hyperbolic
group $G_C$.  To do this, we need to understand conjugators of finite
order elements.  The conjugator length function for finite order
elements of hierarchically hyperbolic groups is unknown, hence this
second step is currently out of reach.
\end{rem}

\bibliographystyle{alpha}
\bibliography{hhs_2015}

\end{document}

%% file: conjugators1.pdf_tex
\begingroup%
  \makeatletter%
  \providecommand\color[2][]{%
    \errmessage{(Inkscape) Color is used for the text in Inkscape, but the package 'color.sty' is not loaded}%
    \renewcommand\color[2][]{}%
  }%
  \providecommand\transparent[1]{%
    \errmessage{(Inkscape) Transparency is used (non-zero) for the text in Inkscape, but the package 'transparent.sty' is not loaded}%
    \renewcommand\transparent[1]{}%
  }%
  \providecommand\rotatebox[2]{#2}%
  \newcommand*\fsize{\dimexpr\f@size pt\relax}%
  \newcommand*\lineheight[1]{\fontsize{\fsize}{#1\fsize}\selectfont}%
  \ifx\svgwidth\undefined%
    \setlength{\unitlength}{269.13714499bp}%
    \ifx\svgscale\undefined%
      \relax%
    \else%
      \setlength{\unitlength}{\unitlength * \real{\svgscale}}%
    \fi%
  \else%
    \setlength{\unitlength}{\svgwidth}%
  \fi%
  \global\let\svgwidth\undefined%
  \global\let\svgscale\undefined%
  \makeatother%
  \begin{picture}(1,0.48328868)%
    \lineheight{1}%
    \setlength\tabcolsep{0pt}%
    \put(0,0){\includegraphics[width=\unitlength,page=1]{conjugators1.pdf}}%
    \put(0.70164248,0.44897767){\makebox(0,0)[lt]{\lineheight{1.25}\smash{\begin{tabular}[t]{l}$y$\end{tabular}}}}%
    \put(0.26634179,0.07567179){\makebox(0,0)[lt]{\lineheight{1.25}\smash{\begin{tabular}[t]{l}$x'$\end{tabular}}}}%
    \put(0.6263153,0.07567179){\makebox(0,0)[lt]{\lineheight{1.25}\smash{\begin{tabular}[t]{l}$y'$\end{tabular}}}}%
    \put(0.84763238,0.06767241){\makebox(0,0)[lt]{\lineheight{1.25}\smash{\begin{tabular}[t]{l}$\mathbf P_U$\end{tabular}}}}%
    \put(0.2776735,0.44897767){\makebox(0,0)[lt]{\lineheight{1.25}\smash{\begin{tabular}[t]{l}$x$\end{tabular}}}}%
  \end{picture}%
\endgroup%

%% file: conjugators2.pdf_tex
\begingroup%
  \makeatletter%
  \providecommand\color[2][]{%
    \errmessage{(Inkscape) Color is used for the text in Inkscape, but the package 'color.sty' is not loaded}%
    \renewcommand\color[2][]{}%
  }%
  \providecommand\transparent[1]{%
    \errmessage{(Inkscape) Transparency is used (non-zero) for the text in Inkscape, but the package 'transparent.sty' is not loaded}%
    \renewcommand\transparent[1]{}%
  }%
  \providecommand\rotatebox[2]{#2}%
  \newcommand*\fsize{\dimexpr\f@size pt\relax}%
  \newcommand*\lineheight[1]{\fontsize{\fsize}{#1\fsize}\selectfont}%
  \ifx\svgwidth\undefined%
    \setlength{\unitlength}{332.90462403bp}%
    \ifx\svgscale\undefined%
      \relax%
    \else%
      \setlength{\unitlength}{\unitlength * \real{\svgscale}}%
    \fi%
  \else%
    \setlength{\unitlength}{\svgwidth}%
  \fi%
  \global\let\svgwidth\undefined%
  \global\let\svgscale\undefined%
  \makeatother%
  \begin{picture}(1,0.57256248)%
    \lineheight{1}%
    \setlength\tabcolsep{0pt}%
    \put(0,0){\includegraphics[width=\unitlength,page=1]{conjugators2.pdf}}%
    \put(0.03705172,0.2075644){\makebox(0,0)[lt]{\lineheight{1.25}\smash{\begin{tabular}[t]{l}$\pi_S(1)$\end{tabular}}}}%
    \put(0.13148438,0.13554989){\makebox(0,0)[lt]{\lineheight{1.25}\smash{\begin{tabular}[t]{l}$x$\end{tabular}}}}%
    \put(0.2177112,0.15862464){\makebox(0,0)[lt]{\lineheight{1.25}\smash{\begin{tabular}[t]{l}$z'$\end{tabular}}}}%
    \put(0.23313719,0.23168896){\makebox(0,0)[lt]{\lineheight{1.25}\smash{\begin{tabular}[t]{l}$z$\end{tabular}}}}%
    \put(0,0){\includegraphics[width=\unitlength,page=2]{conjugators2.pdf}}%
    \put(0.30708615,0.07889137){\makebox(0,0)[lt]{\lineheight{1.25}\smash{\begin{tabular}[t]{l}$\leq L$\end{tabular}}}}%
    \put(0.33065622,0.25578166){\makebox(0,0)[lt]{\lineheight{1.25}\smash{\begin{tabular}[t]{l}$w$\end{tabular}}}}%
    \put(0.33065622,0.33593615){\makebox(0,0)[lt]{\lineheight{1.25}\smash{\begin{tabular}[t]{l}$\leq L$\end{tabular}}}}%
    \put(0.16432272,0.29584095){\makebox(0,0)[lt]{\lineheight{1.25}\smash{\begin{tabular}[t]{l}$\leq \tau(b)$\end{tabular}}}}%
    \put(0.32215497,0.40394606){\makebox(0,0)[lt]{\lineheight{1.25}\smash{\begin{tabular}[t]{l}$w'$\end{tabular}}}}%
    \put(0.19827982,0.37844235){\makebox(0,0)[lt]{\lineheight{1.25}\smash{\begin{tabular}[t]{l}$gz'$\end{tabular}}}}%
    \put(0.19949428,0.46952701){\makebox(0,0)[lt]{\lineheight{1.25}\smash{\begin{tabular}[t]{l}$gz$\end{tabular}}}}%
    \put(0,0){\includegraphics[width=\unitlength,page=3]{conjugators2.pdf}}%
    \put(0.30014666,0.5448237){\makebox(0,0)[lt]{\lineheight{1.25}\smash{\begin{tabular}[t]{l}$\leq L$\end{tabular}}}}%
    \put(0.10355176,0.47195589){\makebox(0,0)[lt]{\lineheight{1.25}\smash{\begin{tabular}[t]{l}$\pi_S(g)$\end{tabular}}}}%
    \put(0.14848684,0.38087127){\makebox(0,0)[lt]{\lineheight{1.25}\smash{\begin{tabular}[t]{l}$gx$\end{tabular}}}}%
    \put(0.07197572,0.40030268){\makebox(0,0)[lt]{\lineheight{1.25}\smash{\begin{tabular}[t]{l}$y$\end{tabular}}}}%
    \put(0.62455606,0.54360924){\makebox(0,0)[lt]{\lineheight{1.25}\smash{\begin{tabular}[t]{l}$\pi_S(gb^K)=\pi_S(a^Kg)$\end{tabular}}}}%
    \put(0.79336635,0.28492874){\makebox(0,0)[lt]{\lineheight{1.25}\smash{\begin{tabular}[t]{l}$\pi_S(b^K)$\end{tabular}}}}%
    \put(0.85530399,0.43066422){\makebox(0,0)[lt]{\lineheight{1.25}\smash{\begin{tabular}[t]{l}$\gamma_b$\end{tabular}}}}%
    \put(0.7387156,0.05175199){\makebox(0,0)[lt]{\lineheight{1.25}\smash{\begin{tabular}[t]{l}$\gamma_a$\end{tabular}}}}%
    \put(0.71806976,0.16955487){\makebox(0,0)[lt]{\lineheight{1.25}\smash{\begin{tabular}[t]{l}$\pi_S(a^K)$\end{tabular}}}}%
    \put(0,0){\includegraphics[width=\unitlength,page=4]{conjugators2.pdf}}%
  \end{picture}%
\endgroup%

%% file: conjugators3.pdf_tex
\begingroup%
  \makeatletter%
  \providecommand\color[2][]{%
    \errmessage{(Inkscape) Color is used for the text in Inkscape, but the package 'color.sty' is not loaded}%
    \renewcommand\color[2][]{}%
  }%
  \providecommand\transparent[1]{%
    \errmessage{(Inkscape) Transparency is used (non-zero) for the text in Inkscape, but the package 'transparent.sty' is not loaded}%
    \renewcommand\transparent[1]{}%
  }%
  \providecommand\rotatebox[2]{#2}%
  \newcommand*\fsize{\dimexpr\f@size pt\relax}%
  \newcommand*\lineheight[1]{\fontsize{\fsize}{#1\fsize}\selectfont}%
  \ifx\svgwidth\undefined%
    \setlength{\unitlength}{482.54450479bp}%
    \ifx\svgscale\undefined%
      \relax%
    \else%
      \setlength{\unitlength}{\unitlength * \real{\svgscale}}%
    \fi%
  \else%
    \setlength{\unitlength}{\svgwidth}%
  \fi%
  \global\let\svgwidth\undefined%
  \global\let\svgscale\undefined%
  \makeatother%
  \begin{picture}(1,0.27393439)%
    \lineheight{1}%
    \setlength\tabcolsep{0pt}%
    \put(0,0){\includegraphics[width=\unitlength,page=1]{conjugators3.pdf}}%
    \put(-0.00141103,0.20334675){\makebox(0,0)[lt]{\lineheight{1.25}\smash{\begin{tabular}[t]{l}$\pi_{g^{-1}Z}(1)$\end{tabular}}}}%
    \put(0.03826543,0.14491009){\makebox(0,0)[lt]{\lineheight{1.25}\smash{\begin{tabular}[t]{l}$x$\end{tabular}}}}%
    \put(0.08159554,0.14604503){\makebox(0,0)[lt]{\lineheight{1.25}\smash{\begin{tabular}[t]{l}$z'$\end{tabular}}}}%
    \put(0,0){\includegraphics[width=\unitlength,page=2]{conjugators3.pdf}}%
    \put(0.11992419,0.11585088){\makebox(0,0)[lt]{\lineheight{1.25}\smash{\begin{tabular}[t]{l}$\leq L$\end{tabular}}}}%
    \put(0.29455394,0.12356377){\makebox(0,0)[lt]{\lineheight{1.25}\smash{\begin{tabular}[t]{l}$\gamma^{g^{-1}Z}_a$\end{tabular}}}}%
    \put(0.25456384,0.19279271){\makebox(0,0)[lt]{\lineheight{1.25}\smash{\begin{tabular}[t]{l}$\pi_{g^{-1}Z}(a^K)$\end{tabular}}}}%
    \put(0,0){\includegraphics[width=\unitlength,page=3]{conjugators3.pdf}}%
    \put(0.01184809,0.25884979){\makebox(0,0)[lt]{\lineheight{1.25}\smash{\begin{tabular}[t]{l}$C(g^{-1}Z)$\end{tabular}}}}%
    \put(0,0){\includegraphics[width=\unitlength,page=4]{conjugators3.pdf}}%
    \put(0.08728177,0.20245316){\makebox(0,0)[lt]{\lineheight{1.25}\smash{\begin{tabular}[t]{l}$z$\end{tabular}}}}%
    \put(0.68902308,0.03689096){\makebox(0,0)[lt]{\lineheight{1.25}\smash{\begin{tabular}[t]{l}$w$\end{tabular}}}}%
    \put(0.696632,0.09491186){\makebox(0,0)[lt]{\lineheight{1.25}\smash{\begin{tabular}[t]{l}$\leq L$\end{tabular}}}}%
    \put(0.59388063,0.07286635){\makebox(0,0)[lt]{\lineheight{1.25}\smash{\begin{tabular}[t]{l}$\leq \tau_Z(b)$\end{tabular}}}}%
    \put(0.69138044,0.13230567){\makebox(0,0)[lt]{\lineheight{1.25}\smash{\begin{tabular}[t]{l}$w'$\end{tabular}}}}%
    \put(0.61232117,0.11067409){\makebox(0,0)[lt]{\lineheight{1.25}\smash{\begin{tabular}[t]{l}$gz'$\end{tabular}}}}%
    \put(0.62702101,0.16582766){\makebox(0,0)[lt]{\lineheight{1.25}\smash{\begin{tabular}[t]{l}$gz$\end{tabular}}}}%
    \put(0,0){\includegraphics[width=\unitlength,page=5]{conjugators3.pdf}}%
    \put(0.65987091,0.20976427){\makebox(0,0)[lt]{\lineheight{1.25}\smash{\begin{tabular}[t]{l}$\leq L$\end{tabular}}}}%
    \put(0.55760793,0.1709676){\makebox(0,0)[lt]{\lineheight{1.25}\smash{\begin{tabular}[t]{l}$\pi_Z(g)$\end{tabular}}}}%
    \put(0.5840981,0.11961848){\makebox(0,0)[lt]{\lineheight{1.25}\smash{\begin{tabular}[t]{l}$gx$\end{tabular}}}}%
    \put(0.53683392,0.13030242){\makebox(0,0)[lt]{\lineheight{1.25}\smash{\begin{tabular}[t]{l}$y$\end{tabular}}}}%
    \put(0.81805053,0.18693532){\makebox(0,0)[lt]{\lineheight{1.25}\smash{\begin{tabular}[t]{l}$\pi_Z(gb^K)=\pi_Z(a^Kg)$\end{tabular}}}}%
    \put(0.89210906,0.07078776){\makebox(0,0)[lt]{\lineheight{1.25}\smash{\begin{tabular}[t]{l}$\pi_Z(b^K)$\end{tabular}}}}%
    \put(0.92800955,0.14046144){\makebox(0,0)[lt]{\lineheight{1.25}\smash{\begin{tabular}[t]{l}$\gamma^Z_b$\end{tabular}}}}%
    \put(0,0){\includegraphics[width=\unitlength,page=6]{conjugators3.pdf}}%
    \put(0.53408156,0.25229342){\makebox(0,0)[lt]{\lineheight{1.25}\smash{\begin{tabular}[t]{l}$CZ$\end{tabular}}}}%
    \put(0,0){\includegraphics[width=\unitlength,page=7]{conjugators3.pdf}}%
    \put(0.54791819,0.00398708){\makebox(0,0)[lt]{\lineheight{1.25}\smash{\begin{tabular}[t]{l}$\pi_Z(1)$\end{tabular}}}}%
    \put(0,0){\includegraphics[width=\unitlength,page=8]{conjugators3.pdf}}%
    \put(0.39918145,0.18493664){\makebox(0,0)[lt]{\lineheight{1.25}\smash{\begin{tabular}[t]{l}$g$\end{tabular}}}}%
  \end{picture}%
\endgroup%

%% file: conjugators4.pdf_tex
\begingroup%
  \makeatletter%
  \providecommand\color[2][]{%
    \errmessage{(Inkscape) Color is used for the text in Inkscape, but the package 'color.sty' is not loaded}%
    \renewcommand\color[2][]{}%
  }%
  \providecommand\transparent[1]{%
    \errmessage{(Inkscape) Transparency is used (non-zero) for the text in Inkscape, but the package 'transparent.sty' is not loaded}%
    \renewcommand\transparent[1]{}%
  }%
  \providecommand\rotatebox[2]{#2}%
  \newcommand*\fsize{\dimexpr\f@size pt\relax}%
  \newcommand*\lineheight[1]{\fontsize{\fsize}{#1\fsize}\selectfont}%
  \ifx\svgwidth\undefined%
    \setlength{\unitlength}{269.02426221bp}%
    \ifx\svgscale\undefined%
      \relax%
    \else%
      \setlength{\unitlength}{\unitlength * \real{\svgscale}}%
    \fi%
  \else%
    \setlength{\unitlength}{\svgwidth}%
  \fi%
  \global\let\svgwidth\undefined%
  \global\let\svgscale\undefined%
  \makeatother%
  \begin{picture}(1,0.51023324)%
    \lineheight{1}%
    \setlength\tabcolsep{0pt}%
    \put(0,0){\includegraphics[width=\unitlength,page=1]{conjugators4.pdf}}%
    \put(0.48214139,0.27040752){\makebox(0,0)[lt]{\lineheight{1.25}\smash{\begin{tabular}[t]{l}$\leq E+\sigma$\end{tabular}}}}%
    \put(0.47398626,0.18114152){\makebox(0,0)[lt]{\lineheight{1.25}\smash{\begin{tabular}[t]{l}$\leq E+\sigma$\end{tabular}}}}%
    \put(0.00975806,0.03423095){\makebox(0,0)[lt]{\lineheight{1.25}\smash{\begin{tabular}[t]{l}$\pi_U(1)$\end{tabular}}}}%
    \put(-0.00263828,0.48298937){\makebox(0,0)[lt]{\lineheight{1.25}\smash{\begin{tabular}[t]{l}$\pi_U(g)$\end{tabular}}}}%
    \put(0.78543826,0.47843732){\makebox(0,0)[lt]{\lineheight{1.25}\smash{\begin{tabular}[t]{l}$\pi_U(b^Kg)=\pi_U(ga^K)$\end{tabular}}}}%
    \put(0.91900225,0.00237404){\makebox(0,0)[lt]{\lineheight{1.25}\smash{\begin{tabular}[t]{l}$\pi_U(b^K)$\end{tabular}}}}%
    \put(0.39426123,0.22135489){\makebox(0,0)[lt]{\lineheight{1.25}\smash{\begin{tabular}[t]{l}$\rho^Z_U$\end{tabular}}}}%
  \end{picture}%
\endgroup%

%% file: conjugator.pdf_tex
\begingroup%
  \makeatletter%
  \providecommand\color[2][]{%
    \errmessage{(Inkscape) Color is used for the text in Inkscape, but the package 'color.sty' is not loaded}%
    \renewcommand\color[2][]{}%
  }%
  \providecommand\transparent[1]{%
    \errmessage{(Inkscape) Transparency is used (non-zero) for the text in Inkscape, but the package 'transparent.sty' is not loaded}%
    \renewcommand\transparent[1]{}%
  }%
  \providecommand\rotatebox[2]{#2}%
  \newcommand*\fsize{\dimexpr\f@size pt\relax}%
  \newcommand*\lineheight[1]{\fontsize{\fsize}{#1\fsize}\selectfont}%
  \ifx\svgwidth\undefined%
    \setlength{\unitlength}{509.69777653bp}%
    \ifx\svgscale\undefined%
      \relax%
    \else%
      \setlength{\unitlength}{\unitlength * \real{\svgscale}}%
    \fi%
  \else%
    \setlength{\unitlength}{\svgwidth}%
  \fi%
  \global\let\svgwidth\undefined%
  \global\let\svgscale\undefined%
  \makeatother%
  \begin{picture}(1,0.40281918)%
    \lineheight{1}%
    \setlength\tabcolsep{0pt}%
    \put(0,0){\includegraphics[width=\unitlength,page=1]{conjugator.pdf}}%
    \put(0.06260225,0.24558604){\color[rgb]{0,0,0}\makebox(0,0)[lt]{\lineheight{0}\smash{\begin{tabular}[t]{l}$1$\end{tabular}}}}%
    \put(0.40702401,0.12592438){\color[rgb]{0,0,0}\makebox(0,0)[lt]{\lineheight{0}\smash{\begin{tabular}[t]{l}$\nu$\end{tabular}}}}%
    \put(0.56104475,0.31511421){\color[rgb]{0,0,0}\makebox(0,0)[lt]{\lineheight{0}\smash{\begin{tabular}[t]{l}$b^K$\end{tabular}}}}%
    \put(0.69064528,0.15196726){\color[rgb]{0,0,0}\makebox(0,0)[lt]{\lineheight{0}\smash{\begin{tabular}[t]{l}$b^K\nu$\end{tabular}}}}%
    \put(0.69171669,0.38470183){\color[rgb]{0,0,0}\makebox(0,0)[lt]{\lineheight{0}\smash{\begin{tabular}[t]{l}$b^Kg=ga^K$\end{tabular}}}}%
    \put(0.28776417,0.33566848){\color[rgb]{0,0,0}\makebox(0,0)[lt]{\lineheight{0}\smash{\begin{tabular}[t]{l}$g$\end{tabular}}}}%
    \put(0.76228135,0.06062785){\color[rgb]{0,0,0}\makebox(0,0)[lt]{\lineheight{0}\smash{\begin{tabular}[t]{l}$\mathbf P_{Z}$\end{tabular}}}}%
    \put(0,0){\includegraphics[width=\unitlength,page=2]{conjugator.pdf}}%
    \put(0.1849491,0.01024605){\makebox(0,0)[lt]{\lineheight{1.25}\smash{\begin{tabular}[t]{l}$\xi$\end{tabular}}}}%
    \put(0.50928058,0.02285894){\makebox(0,0)[lt]{\lineheight{1.25}\smash{\begin{tabular}[t]{l}$b^K\xi$\end{tabular}}}}%
  \end{picture}%
\endgroup%

%% file: conjugators5.pdf_tex
\begingroup%
  \makeatletter%
  \providecommand\color[2][]{%
    \errmessage{(Inkscape) Color is used for the text in Inkscape, but the package 'color.sty' is not loaded}%
    \renewcommand\color[2][]{}%
  }%
  \providecommand\transparent[1]{%
    \errmessage{(Inkscape) Transparency is used (non-zero) for the text in Inkscape, but the package 'transparent.sty' is not loaded}%
    \renewcommand\transparent[1]{}%
  }%
  \providecommand\rotatebox[2]{#2}%
  \newcommand*\fsize{\dimexpr\f@size pt\relax}%
  \newcommand*\lineheight[1]{\fontsize{\fsize}{#1\fsize}\selectfont}%
  \ifx\svgwidth\undefined%
    \setlength{\unitlength}{239.0998236bp}%
    \ifx\svgscale\undefined%
      \relax%
    \else%
      \setlength{\unitlength}{\unitlength * \real{\svgscale}}%
    \fi%
  \else%
    \setlength{\unitlength}{\svgwidth}%
  \fi%
  \global\let\svgwidth\undefined%
  \global\let\svgscale\undefined%
  \makeatother%
  \begin{picture}(1,0.49595214)%
    \lineheight{1}%
    \setlength\tabcolsep{0pt}%
    \put(0,0){\includegraphics[width=\unitlength,page=1]{conjugators5.pdf}}%
    \put(0.1905854,0.45733075){\makebox(0,0)[lt]{\lineheight{1.25}\smash{\begin{tabular}[t]{l}$\pi_Z(g)$\end{tabular}}}}%
    \put(0.67395154,0.46978864){\makebox(0,0)[lt]{\lineheight{1.25}\smash{\begin{tabular}[t]{l}$\pi_Z(b^Kg)=\pi_Z(ga^K)$\end{tabular}}}}%
    \put(0.84088732,0.25800447){\makebox(0,0)[lt]{\lineheight{1.25}\smash{\begin{tabular}[t]{l}$\pi_Z(b^K\nu)$\end{tabular}}}}%
    \put(0.90566839,0.36015919){\makebox(0,0)[lt]{\lineheight{1.25}\smash{\begin{tabular}[t]{l}$\leq E+A$\end{tabular}}}}%
    \put(0.29024851,0.28292027){\makebox(0,0)[lt]{\lineheight{1.25}\smash{\begin{tabular}[t]{l}$\pi_Z(\mu(g,\nu))$\end{tabular}}}}%
    \put(0.22297591,0.08110244){\makebox(0,0)[lt]{\lineheight{1.25}\smash{\begin{tabular}[t]{l}$\pi_Z(\nu)$\end{tabular}}}}%
    \put(0.04109068,0.24554661){\makebox(0,0)[lt]{\lineheight{1.25}\smash{\begin{tabular}[t]{l}$\rho^U_Z$\end{tabular}}}}%
    \put(0.09092226,0.00884665){\makebox(0,0)[lt]{\lineheight{1.25}\smash{\begin{tabular}[t]{l}$\leq E$\end{tabular}}}}%
    \put(-0.00126614,0.36265078){\makebox(0,0)[lt]{\lineheight{1.25}\smash{\begin{tabular}[t]{l}$\leq 2 E+\sigma$\end{tabular}}}}%
    \put(0.4596758,0.44487286){\makebox(0,0)[lt]{\lineheight{1.25}\smash{\begin{tabular}[t]{l}$\geq KT$\end{tabular}}}}%
  \end{picture}%
\endgroup%

%% file: ConjugatorLength_17_FINAL_ArXiv.bbl
\def\cprime{$'$}
\begin{thebibliography}{BHS17b}

\bibitem[ABD21]{AbbottBehrstock:HHSlargest}
Carolyn Abbott, Jason Behrstock, and Matthew~Gentry Durham.
\newblock Largest acylindrical actions and {S}tability in hierarchically
  hyperbolic groups.
\newblock {\em Trans. Amer. Math. Soc. Ser. B}, 8:66--104, 2021.

\bibitem[BD14]{BehrstockDrutu:thick2}
J.~Behrstock and C.~Dru\c{t}u.
\newblock Divergence, thick groups, and short conjugators.
\newblock {\em Illinois J. Math.}, 58(4):939--980, 2014.

\bibitem[BH99]{BridsonHaefliger}
Martin~R. Bridson and Andr{\'e} Haefliger.
\newblock {\em Metric spaces of non-positive curvature}.
\newblock Springer-Verlag, Berlin, 1999.

\bibitem[BHS17a]{BehrstockHagenSisto:asdim}
J.~Behrstock, M.F. Hagen, and A.~Sisto.
\newblock {Asymptotic dimension and small-cancellation for hierarchically
  hyperbolic spaces and groups}.
\newblock {\em Proc. London Math. Soc.}, 114(5):890--926, 2017.

\bibitem[BHS17b]{BehrstockHagenSisto:HHS1}
Jason Behrstock, Mark~F Hagen, and Alessandro Sisto.
\newblock Hierarchically hyperbolic spaces {I}: curve complexes for cubical
  groups.
\newblock {\em Geometry \& Topology}, 21:1731--1804, 2017.

\bibitem[BHS19]{BehrstockHagenSisto:HHS2}
Jason Behrstock, Mark~F Hagen, and Alessandro Sisto.
\newblock Hierarchically hyperbolic spaces {II}: combination theorems and the
  distance formula.
\newblock {\em Pacific J. Math.}, 299(2):257--338, 2019.

\bibitem[BHS21]{BehrstockHagenSisto:quasiflats}
Jason Behrstock, Mark~F Hagen, and Alessandro Sisto.
\newblock Quasiflats in hierarchically hyperbolic spaces.
\newblock {\em Duke Math. J.}, 170(5):909--996, 2021.

\bibitem[BLM83]{BirmanLubotzkyMcCarthy}
J.~Birman, A.~Lubotzky, and J.~McCarthy.
\newblock Abelian and solvable subgroups of the mapping class groups.
\newblock {\em Duke Math. J.}, 50(4):1107--1120, 1983.

\bibitem[BM97]{BurgerMozes:TreeProducts}
M.~Burger and S.~Mozes.
\newblock Finitely presented simple groups and products of trees.
\newblock {\em C. R. Acad. Sci. Paris S\'er. I Math.}, 324(7):747--752, 1997.

\bibitem[BM00]{BurgerMozes:TreeProductsIHES}
Marc Burger and Shahar Mozes.
\newblock Lattices in product of trees.
\newblock {\em Inst. Hautes \'Etudes Sci. Publ. Math.}, (92):151--194 (2001),
  2000.

\bibitem[Bow08]{Bowditch:tight}
Brian~H. Bowditch.
\newblock Tight geodesics in the curve complex.
\newblock {\em Inventiones mathematicae}, 171(2):281--300, 2008.

\bibitem[BR]{BerlyneRussell:HHSgraphproducts}
Daniel Berlyne and Jacob Russell.
\newblock Hierarchical hyperbolicity of graph products.
\newblock {\em {Groups, Geometry, and Dynamics}}.
\newblock To appear.

\bibitem[BR20]{BerlaiRobbio:combinationHHG}
Federico Berlai and Bruno Robbio.
\newblock A refined combination theorem for hierarchically hyperbolic groups.
\newblock {\em Groups Geom. Dyn.}, 14(4):1127--1203, 2020.

\bibitem[Bum15]{Bumagin:relhyplinconj}
Inna Bumagin.
\newblock Time complexity of the conjugacy problem in relatively hyperbolic
  groups.
\newblock {\em Internat. J. Algebra Comput.}, 25(5):689--723, 2015.

\bibitem[Cap17]{Caprace:indiscrete}
Pierre-Emmanuel Caprace.
\newblock Finite and infinite quotients of discrete and indiscrete groups.
\newblock Preprint, \textsc{arXiv:1709.05949}, 2017.

\bibitem[CGW09]{CrispGodelleWiest:ConjRAAGs}
John Crisp, Eddy Godelle, and Bert Wiest.
\newblock The conjugacy problem in subgroups of right-angled {A}rtin groups.
\newblock {\em J. Topol.}, 2(3):442--460, 2009.

\bibitem[Cou16]{Coulon}
R\'{e}mi~B. Coulon.
\newblock Partial periodic quotients of groups acting on a hyperbolic space.
\newblock {\em Ann. Inst. Fourier (Grenoble)}, 66(5):1773--1857, 2016.

\bibitem[DHS17]{DurhamHagenSisto:HHS_boundary}
Matthew Durham, Mark Hagen, and Alessandro Sisto.
\newblock Boundaries and automorphisms of hierarchically hyperbolic spaces.
\newblock {\em Geom. Topol.}, 21(6):3659--3758, 2017.

\bibitem[DHS20]{DurhamHagenSisto:HHS_boundary_symphoria}
Matthew~Gentry Durham, Mark~F. Hagen, and Alessandro Sisto.
\newblock Correction to the article {B}oundaries and automorphisms of
  hierarchically hyperbolic spaces.
\newblock {\em Geom. Topol.}, 24(2):1051--1073, 2020.

\bibitem[HHP20]{HHP:semihyp}
Thomas Haettel, Nima Hoda, and Harry Petyt.
\newblock The coarse {H}elly property, hierarchical hyperbolicity, and
  semihyperbolicity.
\newblock {\em ar{X}iv:2009.14053}, 2020.

\bibitem[HS20]{HagenSusse:HHScubical}
Mark~F. Hagen and Tim Susse.
\newblock On hierarchical hyperbolicity of cubical groups.
\newblock {\em Israel J. Math.}, 236(1):45--89, 2020.

\bibitem[HW08]{HaglundWiseSpecial}
Fr{\'e}d{\'e}ric Haglund and Daniel~T. Wise.
\newblock Special cube complexes.
\newblock {\em Geom. Funct. Anal.}, 17(5):1 551--1620, 2008.

\bibitem[JW09]{JanzenWise}
David Janzen and Daniel~T. Wise.
\newblock A smallest irreducible lattice in the product of trees.
\newblock {\em Algebr. Geom. Topol.}, 9(4):2191--2201, 2009.

\bibitem[Lys89]{Lysenok}
I.~G. Lys\"enok.
\newblock Some algorithmic properties of hyperbolic groups.
\newblock {\em Izv. Akad. Nauk SSSR Ser. Mat.}, 53(4):814--832, 912, 1989.

\bibitem[MM00]{MasurMinsky:II}
Howard~A Masur and Yair~N Minsky.
\newblock Geometry of the complex of curves {II}: Hierarchical structure.
\newblock {\em Geometric and Functional Analysis}, 10(4):902--974, 2000.

\bibitem[Rat07]{Rattaggi}
Diego Rattaggi.
\newblock A finitely presented torsion-free simple group.
\newblock {\em J. Group Theory}, 10(3):363--371, 2007.

\bibitem[RS]{RobbioSpriano:2decomposableHHG}
B.~Robbio and S.~Spriano.
\newblock Hierarchical hyperbolicity of hyperbolic-2-decomposable groups.
\newblock {\em {2007.13383}}.

\bibitem[RST18]{RussellSprianoTran:convexity}
Jacob Russell, Davide Spriano, and Hung~C. Tran.
\newblock Convexity in hierarchically hyperbolic spaces.
\newblock {\em Algebr. Geom. Topol.}, arXiv:1809.09303, 2018.
\newblock To appear.

\bibitem[Rus22]{Russell:relative}
Jacob Russell.
\newblock From hierarchical to relative hyperbolicity.
\newblock {\em Int. Math. Res. Not. IMRN}, (1):575--624, 2022.

\bibitem[Sal14]{Sale:LinConjLie}
Andrew Sale.
\newblock Bounded conjugators for real hyperbolic and unipotent elements in
  semisimple {L}ie groups.
\newblock {\em J. Lie Theory}, 24(1):259--305, 2014.

\bibitem[Sal16]{Sale:GeomConjLamp}
Andrew Sale.
\newblock Geometry of the conjugacy problem in lamplighter groups.
\newblock In {\em Algebra and computer science}, volume 677 of {\em Contemp.
  Math.}, pages 171--183. Amer. Math. Soc., Providence, RI, 2016.

\bibitem[Ser89]{Servatius:AutRAAG}
Herman Servatius.
\newblock Automorphisms of graph groups.
\newblock {\em J. Algebra}, 126(1):34--60, 1989.

\bibitem[Spr18]{Spriano:GraphsHHG}
D.~Spriano.
\newblock {Hyperbolic HHS II: Graphs of hierarchically hyperbolic groups}.
\newblock arXiv:1801.01850, 2018.

\bibitem[Tao13]{Tao:LinConjMCG}
Jing Tao.
\newblock Linearly bounded conjugator property for mapping class groups.
\newblock {\em Geom. Funct. Anal.}, 23(1):415--466, 2013.

\bibitem[Wis07]{Wise:CSC}
Daniel~T Wise.
\newblock Complete square complexes.
\newblock {\em Commentarii Mathematici Helvetici}, 82(4):683--724, 2007.

\end{thebibliography}
